\documentclass[reqno]{amsart}
\usepackage{amsfonts,amsmath,amssymb,dsfont}
\usepackage{color}
\usepackage[utf8]{inputenc}  
\usepackage[T1]{fontenc}  
\usepackage{graphicx}

\numberwithin{equation}{section}

\usepackage[colorinlistoftodos]{todonotes}

\usepackage{xcolor}

\newcommand{\R}{\mathbb{R}}

\newcommand{\ve}{\varepsilon}

\newcommand{\de}{\partial}
\newcommand{\n}{\nabla}

\newcommand{\N}{\mathbb{N}}

\renewcommand{\(}{\left(}
\renewcommand{\)}{\right)}
\newcommand{\M}[1]{\mathcal{#1}}

\newtheorem{thm}{Theorem}[section]
\newtheorem*{thm*}{Theorem}
\newtheorem{lem}{Lemma}[section]

\newtheorem{cor}{Corollary}[section]
\newtheorem*{cor*}{Corollary}
\newtheorem{prop}{Proposition}[section]

\newtheorem{rem}{Remark}[section]

\newtheorem*{opbm}{Open problem}

\begin{document}
\title[Critical points of the Trudinger-Moser functional]{Critical points of arbitrary energy for the Trudinger-Moser functional in planar domains}

\begin{abstract}
Given a smoothly bounded non-contractible domain $\Omega\subset \mathbb{R}^2$, we prove the existence of positive critical points of the Trudinger-Moser embedding for arbitrary Dirichlet energies. This is done via degree theory, sharp compactness estimates and a topological argument relying on the Poincar\'e-Hopf theorem.
\end{abstract}

\author[A. Malchiodi]{Andrea Malchiodi}\thanks{A.M. is supported by the project \emph{Geometric problems with loss of compactness} from Scuola Normale Superiore and by the PRIN Project 2022AKNSE4 \emph{Variational and Analytical aspects of Geometric PDE}. He is also member of GNAMPA as part of INdAM}
\address[Andrea Malchiodi]{Scuola Normale Superiore, Piazza dei Cavalieri 7, 56126 Pisa, Italy}
\email{andrea.malchiodi@sns.it}

\author[L. Martinazzi]{Luca Martinazzi}\thanks{L.M. is supported by the Swiss National Science Foundation, grant n. P2BSP2 - 172064, the Fondazione Cariplo, grant n. 2022-2118, and by the PRIN Project 2022PJ9EFL \emph{Geometric Measure Theory: Structure of Singular Measures, Regularity Theory and Applications in the Calculus of Variations}}
\address[Luca Martinazzi]{Dipartimento di Matematica Guido Castelnuovo, Universit\`a La Sapienza, Piazzale Aldo Moro 5, 00185 Roma, Italy}
\email{luca.martinazzi@uniroma1.it}

\author[P.-D. Thizy]{Pierre-Damien Thizy}\thanks{P.T. is supported
by the ANR JCJC BLADE-JC, FRMARA (FR 3490) and SENS 2023 (UCBL)}
\address[Pierre-Damien Thizy]{Institut Camille Jordan, Universit\'e Claude Bernard Lyon 1, B\^atiment Braconnier, 21 avenue Claude Bernard, 69622 Villeurbanne Cedex, France}
\email{pierre-damien.thizy@univ-lyon1.fr}

\maketitle

\section{Introduction}
Given a smoothly bounded domain $\Omega\subset \mathbb{R}^2$, i.e. a bounded domain with smooth boundary, let $H^1_0=H^1_0(\Omega)$ be the usual Sobolev space of functions with zero trace on $\partial\Omega$ endowed with the norm $\|\cdot\|_{H^1_0}$ given by
$$\|u\|_{H^1_0}^2=\int_\Omega |\nabla u|^2 dx.$$
Given any positive real number $\beta>0$, our main purpose is to discuss the existence of nonnegative critical points of the classical 
Moser-Trudinger functional
\begin{equation}\label{MTFunctional}
F(u):=\int_\Omega \left(e^{u^2}-1\right) dx\,,
\end{equation}
constrained to the submanifold
\begin{equation}\label{DefMBeta}
\mathcal{M}_\beta:=\left\{v\in H^1_0~:~\|v\|^2_{H^1_0}=\beta \right\} 
\end{equation}
of $H^1_0$. This is equivalent to finding a solution of the following elliptic problem
\begin{equation}\label{MainEq}\tag{$\mathcal{E}_\beta$}
\begin{cases}
&\Delta u= 2\lambda u e^{u^2}\,,\, u>0\,\text{ in }\Omega\,,\\
&u=0\text{ on }\partial \Omega\,,\\
&\int_\Omega |\nabla u|^2 dx=\beta\,,
\end{cases}
\end{equation}
for some real number $\lambda$, where $\Delta=-\partial_{xx}-\partial_{yy}$ is the (nonnegative) Laplacian.

\medskip

When $\beta>4\pi$, this is a long-standing open problem, that can be traced back to the seminal work of Struwe \cite{StruweCrit}, as we shall briefly explain.

Let us first state the celebrated Moser-Trudinger inequality \cite{MoserIneq, Poho65, TrudingerOrlicz} as
\begin{equation}\label{MTineq}\tag{$\mathcal{MT}$}
\sup_{u\in \mathcal{M}_\beta} F(u)<\infty\quad \text{if and only if }\beta\le 4\pi.
\end{equation}
A simple consequence of \eqref{MTineq} is that for $\beta\in (0,4\pi)$ the functional $F|_{\mathcal{M}_\beta}$ admits a maximizer. The existence of maximizers even in the critical case $\beta= 4\pi$, when the Palais-Smale condition does not hold anymore (see \cite{AdimurthiFailure}), is much more delicate, and was proven by Carleson-Chang \cite{CarlesonChang} when $\Omega$ is a disk and extended to general bounded domains by Struwe \cite{StruweCrit} and Flucher \cite{Flucher}.

The existence of critical points of $F|_{\mathcal{M}_\beta}$ in the supercritical regime $\beta> 4\pi$ is even more delicate, due to the critical nature of the nonlinearity which, in addition to the failure of the Palais-Smale condition for every $\beta\ge 4\pi$ (as shown in \cite{CostaTintarev}), makes the blow-up estimates very subtle and a priori leaves open possibility of no critical points for any $\beta>4\pi$. Instead, relying on the concentration-compactness principle of Lions and on his own monotonicity trick, still in \cite{StruweCrit}, Struwe proved the existence of $\ve_0>0$ such that $F|_{\mathcal{M}_\beta}$ has a positive local maximizer for every $\beta\in (4\pi,4\pi+\ve_0)$ and a mountain pass-type positive critical point for \emph{almost every} $\beta\in (4\pi,4\pi+\ve_0)$. The existence of a second positive critical point for \emph{every} $\beta\in (4\pi,4\pi+\ve_0)$ was proven by Lamm-Robert-Struwe \cite{LammRobertStruwe} using a flow approach. In these results, there is no condition on the smoothly bounded domain $\Omega$, but on the other hand, the requirement that $\beta$ is close to $4\pi$ is essential. This, of course, leaves open the question of the existence of critical points of $F|_{\mathcal{M}_\beta}$ for $\beta$ much larger than $4\pi$.

That this is not just a technical issue is confirmed by the result of the first and second authors \cite{MalchMartJEMS}, proving that when $\Omega$ is a disk  $F|_{\mathcal{M}_\beta}$ has no positive critical point for $\beta$ sufficiently large. 
On the other hand, Del Pino-Musso-Ruf \cite{DelPBeyond}, in the case in which $\Omega$ is not simply connected and $\beta$ is close to $8\pi$, or $\Omega$ is an annulus and $\beta$ is close to $4\pi N$ for some $N\in \mathbb{N}^\star:=\mathbb{N}\setminus\{0\}$, used a Lyapunov-Schmidt construction to show the existence of critical points of $F|_{\mathcal{M}_\beta}$. This suggests that the topology of $\Omega$ plays a crucial role in the solvability of Problem \eqref{MainEq}, and justifies the assumption of our main result, which reads as follows (see also Remark \ref{NonAutonomousVersion} below):
\begin{thm}\label{MainThm}
Let $\Omega\subset \mathbb{R}^2$ be a smoothly bounded non-contractible domain. Then, given any positive real number $\beta>0$, there exists a nonnegative function $u$, critical point of $F$ constrained to $\mathcal{M}_\beta\subset H^1_0$. In particular, $u$ is smooth  and solves \eqref{MainEq}.
\end{thm}

\noindent Theorem \ref{MainThm}, complemented with the previous results when $\Omega$ is a disk, gives a fairly complete answer to the problem of the existence of solutions to \eqref{MainEq}.

The case in which the domain $\Omega$ is replaced by a two-dimensional closed surface was recently studied in \cite{DMMT} using the \emph{barycenter technique} of Djadli-Malchiodi \cite{DjaMal} (see also Bahri \cite{BahriBook}) taking advantage of the variational structure, a minmax argument together with Struwe's monotonicity trick to obtain solutions of a \emph{subcritical} approximate problem for almost every $\beta>0$, and blow-up analysis together with energy estimates at critical energy levels to conclude.

In order to prove Theorem \ref{MainThm}, we will use a different approach, based on the Leray-Schauder degree. This can be divided in four main steps. First, assuming that $\beta\not\in 4\pi\mathbb{N}$, we will deform Problem \eqref{MainEq} to link it to the \emph{Mean-Field} equation
\begin{equation}\label{MeanFieldEquation}\tag{$\mathcal{E}_{\beta}^{MF}$}
\begin{cases}
& \Delta u=\beta \frac{2 e^u}{\int_\Omega e^u dx}\,,\\
& u=0 \text{ on }\partial\Omega\,,
\end{cases}
\end{equation} 
studied e.g. in Ding-Jost-Li-Wang \cite{DJLWA}, and for which a degree theory approach has been initiated by Li \cite{YYLi} and developed by Chen and Lin \cite{ChenLinSharpEst,ChenLin-Liouville}.

More precisely, given $\beta>0$, we consider the $p$-dependent family of problems
\begin{equation}\label{IntermEq}\tag{$\mathcal{E}_{p,\beta}$}
\begin{cases}
&\Delta u= p\lambda u^{p-1} e^{u^p}\,,\, u>0\,\text{ in }\Omega\,,\\
&u=0\text{ on }\partial \Omega\,,\\
&\frac{\lambda p^2}{2} \left(\int_\Omega \left(e^{u^p}-1\right) dx\right)^{\frac{2-p}{p}}\left(\int_\Omega u^p e^{u^p} dx \right)^{\frac{2(p-1)}{p}}=\beta\,,
\end{cases}
\end{equation}
parametrized by $p\in[1,2]$. Notice that $(\mathcal{E}_{2,\beta})$ is equivalent to \eqref{MainEq} thanks to integration by parts, while $(\mathcal{E}_{1,\beta})$ can be easily deformed to \eqref{MeanFieldEquation}.

That these deformations preserve the Leray-Schauder degree is a consequence of the sharp compactness estimates that have been developed in a series of papers (\cite{DMMT,DruetDuke,MalchMartJEMS,MartMan}), and that we can adapt to the case of a bounded domain (Theorem \ref{ThmBUp}), upon ruling out boundary blow-up (thanks to Proposition \ref{PropMovPlane} below). In particular for every $N\in \mathbb{N}$ and $\beta\in (4\pi N, 4\pi (N+1))$ the Leray-Schauder degree $d_{p,\beta}(\Omega)$ is a constant, not depending on $p\in [1,2]$ and equal to the Leray-Schauder degree $d_{N}^{MF}(\Omega)$ of \eqref{MeanFieldEquation}.

Then it suffices to compute $d_N^{MF}(\Omega)$. That $d_0^{MF}(\Omega)=1$ is fairly easy, hence the main issue is to compute the degree jumps $d_{N+1}^{MF}(\Omega)-d_N^{MF}(\Omega)$ for any $N\in\mathbb{N}$, i.e. (roughly speaking) the number of blowing-up families of solutions $(u_\ve)$ to \eqref{MeanFieldEquation} with $\beta_\ve\to 4\pi(N+1)$.  Using a Lyapunov-Schmidt construction (a so-called finite dimensional reduction), Chen-Lin \cite{ChenLin-Liouville} reduced this problem to the computation of the Brouwer degree of a finite dimensional vector field, see Proposition \ref{p:degreejump}.

The third step is to compute the Brouwer degree of such vector field, i.e. Proposition \ref{p:degree}. The proof of this proposition, which is the most delicate part of this work and will be given in Section \ref{SectChenLin}, is based on a new topological argument and a blow-up analysis, which allow us to use the Poincar\'e-Hopf theorem even in the context of bounded domains, when the vector field does not always point outwards. We will state this result in more generality than needed here, because its relevance goes beyond its application to Theorem \ref{MainThm}. Moreover with minor changes, this result can also be carried and applied to higher even dimension (see comments after Proposition \ref{p:degreejump}, and Remark \ref{Odd}).

These 3 steps are sufficient to compute the Leray-Schauder degree of \eqref{MeanFieldEquation}, hence of \eqref{IntermEq} and \eqref{MainEq}, for every positive $\beta\not\in 4\pi \mathbb{N}$. Finally, in order to also include the case $\beta\in 4\pi\mathbb{N^\star}$, we prove that for $p\in (1,2]$ solutions to \eqref{IntermEq} are compact from the left (Theorem \ref{t:4pi+} below).

Collecting these results we finally obtain:

\begin{thm}\label{ThmDegree}
Let $\Omega\subset \mathbb{R}^2$ be a smoothly bounded domain. Let $p\in (1,2]$, $N\in \mathbb{N}$ and $\beta\in (4\pi N, 4\pi(N+1)]$ be given. Then the total Leray-Schauder degree $d_{p,\beta}(\Omega)$ of the solutions of \eqref{IntermEq} is well-defined and equals the binomial number $\binom{N-\chi(\Omega)}{N}$, where $\chi(\Omega)$ stands for the Euler characteristic of $\Omega$. In particular, if either $\Omega$ is not simply connected or if $N=0$, we have $d_{p,\beta}(\Omega)\neq 0$, so that the set $\mathcal{C}_{p,\beta}(\Omega)$ of the solutions of \eqref{IntermEq} is not empty.
\end{thm}

\noindent We adopt here the usual convention that $\binom{-1}{0}=1$ and $\binom{N-1}{N}=0$ for $N\in \mathbb{N}^\star$. We recall that the genus $g(\Omega):=1-\chi(\Omega)$ is basically the ``number of holes'' inside $\Omega$, so that $\Omega$ is simply connected if and only if $\chi(\Omega)=1$. 

\medskip

Interestingly enough, given $p\in[1,2)$, our topological assumption on $\Omega$ is sharp in the following sense: \eqref{IntermEq} does not have any solution for $\beta$ sufficiently large if $\Omega$ is \emph{simply connected} (see \cite{BattagliaUnifBd,BatManThi}). Yet, for $p=2$, we state the following delicate question which looks  open, except in the aforementioned case of the disk:
\begin{opbm}
Let $\Omega\subset \mathbb{R}^2$ be a \emph{simply connected} domain. Does there exists $\beta^\sharp>4\pi$ such that \eqref{MainEq} does not have any solution for $\beta>\beta^\sharp$?
\end{opbm}
\noindent To conclude this introductory part, we mention that the degree formula in Theorem \ref{ThmDegree} also implies the existence of blowing-up solutions for \eqref{IntermEq} (see Remark \ref{RemBupSol}).

\section{Proof of Theorems \ref{MainThm} and \ref{ThmDegree}}
Let us first notice that any nonnegative nontrivial weak $H^1_0$-solution of the PDE in \eqref{IntermEq} is smooth and positive in $\Omega$, as proven from Trudinger \cite{TrudingerOrlicz} with standard elliptic theory (see for instance Gilbarg-Trudinger \cite{Gilbarg}).

As a first ingredient, since the nonlinearities in the right-hand of the PDE in \eqref{IntermEq} are autonomous, we may use the following by-product of the moving plane technique, already used for instance by Adimurthi and Druet \cite{AdimurthiDruet} in a similar context:
\begin{prop}\label{PropMovPlane}
Let $\Omega\subset \mathbb{R}^2$ be a smoothly bounded domain. Then there exists $\delta,\delta'>0$ depending only on $\Omega$ such that, for any nonnegative, nondecreasing $C^1$-function $f:[0,+\infty)\to \mathbb{R}$, for all $v$ solving
\begin{equation*}
\begin{cases}
&\Delta v=f(v), v>0\text{ in }\Omega\,,\\
&v=0\text{ on }\partial\Omega\,,
\end{cases}
\end{equation*}
and all $x\in \partial \Omega$, the function $t\mapsto v(x-t\nu(x))$ has positive derivative in $[0,\delta]$, where $\nu(x)$ is the unit outward normal to $\partial\Omega$ at $x$, and we have
\begin{equation}\label{CritFarBdry}
\nabla v(x)=0 \quad\implies\quad d(x,\partial\Omega)\ge \delta'>0\,.
\end{equation}
\end{prop}

\begin{proof}[Proof of Proposition \ref{PropMovPlane} (completed)] The existence of $\delta>0$ as in the statement of the proposition is a direct consequence of the classical moving plane argument in de Figueiredo-Lions-Nussbaum \cite{deFigLions} if $\Omega$ is convex, while it has to be combined with Kelvin's transform in our general two-dimensional case, as observed by Han \cite{Han}. Then, since the map $(t,x)\mapsto x-t\nu(x)$ is a diffeomorphism from $[0,\delta]\times \partial \Omega$ onto a neighborhood of $\partial \Omega$ inside $\bar{\Omega}$, up to reducing $\delta>0$, \eqref{CritFarBdry} follows by compactness of $\partial\Omega$. 
\end{proof}

Proposition \ref{PropMovPlane} is an important ingredient to take advantage of the blow-up analysis developed to get \cite[Theorem 4.1]{DMMT} on a closed surface in our present setting of a bounded domain $\Omega$, ensuring that no new phenomena arise close to the boundary. Hence we have the following theorem:

\begin{thm}\label{ThmBUp}
 Let $(\lambda_\varepsilon)_\varepsilon$ be any sequence of positive real numbers and $(p_\varepsilon)_\varepsilon$ be any sequence of numbers in $[1,2]$. Let $(u_\varepsilon)_\varepsilon$ be a sequence of smooth functions solving 
\begin{equation}\label{EqBup}
\begin{cases}
&\Delta u_\varepsilon=p_\varepsilon \lambda_\varepsilon  u_\varepsilon^{p_\varepsilon-1} e^{u_\varepsilon^{p_\varepsilon}}\,,\quad u_\varepsilon>0\text{ in }\Omega\,,\\
& u_\varepsilon=0\quad\text{ on }\partial \Omega\,,
\end{cases}
\end{equation}
for all $\varepsilon$. Let $\beta_\varepsilon$ be given by
\begin{equation*}
\beta_\varepsilon=\frac{\lambda_\varepsilon p_\varepsilon^2}{2} \left( \int_\Omega  \left(e^{u_\varepsilon^{p_\varepsilon}}-1\right)~dx\right)^{\frac{2-p_\varepsilon}{p_\varepsilon}}\left(  \int_\Omega u_\varepsilon^{p_\varepsilon} e^{u_\varepsilon^{p_\varepsilon}}~dx\right)^{\frac{2(p_\varepsilon-1)}{p_\varepsilon}}
\end{equation*}
for all $\varepsilon$. If we assume the energy bound
\begin{equation*}
\lim_{\varepsilon\to 0}\beta_\varepsilon=\beta \in [0,+\infty)\,,
\end{equation*}
but the pointwise blow-up of the $u_\varepsilon$'s, namely $\lim_{\varepsilon\to 0} \max_{\Omega} u_\varepsilon=+\infty\,,
$ then $\lambda_\varepsilon\to 0^+$ and there exists an integer $N\in \mathbb{N}^\star$ such that $\beta=4\pi N\,.$
\end{thm}

\noindent Independently, multiplying the equation in \eqref{IntermEq} by a first Dirichlet eigenfunction $v_1$ of the Laplacian, integrating by parts and using $pt^{p-1}e^{t^p}\ge 2t$ for all $t\ge 0$ and $p\in [1,2]$, we get 
\begin{equation}\label{ControlLambda1}
0<2\lambda\le \lambda_1  
\end{equation}
if a (nontrivial) solution $u$ of \eqref{IntermEq} exists. 

\begin{proof}[Outline of the proof of Theorem \ref{ThmBUp}]
Overall, the proof follows closely \cite[Sections 2-4]{DMMT}. It is even simpler in the present framework where Dirichlet boundary conditions are imposed, since we no longer need to handle the additional coercivity term $h u_\varepsilon$ in the left-hand $\Delta_g u_\varepsilon+h u_\varepsilon$ of the equation in \cite[equation (4.1)]{DMMT} analogue to \eqref{EqBup} here: indeed, all along the argument of \cite{DMMT}, this additional term is just estimated by error terms which can be taken now to be zero. 

\medskip

More precisely, as pioneered by Druet \cite[Sections 3-4]{DruetDuke} when $p_\varepsilon=2$ for all $\varepsilon$ (see also \cite[Section 3]{DruThiI}), we start with the following proposition giving a first pointwise exhaustion of the concentration points, as well as weak gradient estimates under the assumptions of Theorem \ref{ThmBUp}:

\begin{prop}\label{PropWeakPwEst}
 Up to a subsequence, there exist an integer $n\in \mathbb{N}^\star$ and sequences $(x_{i,\varepsilon})_{\varepsilon}$ of points in $\Omega$ such that $\nabla u_\varepsilon(x_{i,\varepsilon})=0$, such that, setting $\gamma_{i,\varepsilon}:=u_\varepsilon(x_{i,\varepsilon})$,
\begin{equation}\label{MuToZero}
\mu_{i,\varepsilon}:=\left(\frac{8}{\lambda_\varepsilon p_\varepsilon^2 \gamma_{i,\varepsilon}^{2(p_\varepsilon-1)}  e^{\gamma_{i,\varepsilon}^{p_\varepsilon}}}\right)^{\frac{1}{2}} \to 0\,,
\end{equation} 
and such that
\begin{equation}\label{Bubbling}
\frac{p_\varepsilon}{2}\gamma_{i,\varepsilon}^{p_\varepsilon-1}(\gamma_{i,\varepsilon}-u_\varepsilon(x_{i,\varepsilon}+\mu_{i,\varepsilon}\cdot))\to T_0:=\ln\left(1+|\cdot|^2 \right)\text{ in }C^1_{loc}(\mathbb{R}^2)\,,
\end{equation}
as $\varepsilon\to 0$, for all $i\in\{1,...,n\}$. Moreover, there exist $C>0$ such that we have
\begin{equation*}
\left(\min_{i\in\{1,...,n\}}|x_{i,\varepsilon}-\cdot|\right) |\nabla u_\varepsilon|\, u_\varepsilon^{p_\varepsilon-1}  \le C\text{ in }\Omega
\end{equation*}
for all $\varepsilon$. We also have that $\lim_{\varepsilon\to 0}x_{i,\varepsilon}=x_i$ for all $i$, and that there exists $u_0\in C^2(\bar{\Omega}\backslash \mathcal{S})$ such that 
\begin{equation*}
\lim_{\varepsilon\to 0}u_\varepsilon=u_0\text{ in }C^1_{loc}(\bar{\Omega}\backslash \mathcal{S})\,,
\end{equation*}
where $\mathcal{S}\subset \Omega$ consists of all the $x_i$'s.
\end{prop}

About the concentration points $x_{i,\varepsilon}$ of this proposition, being critical points of the $u_\varepsilon$'s by construction, \eqref{CritFarBdry} yields that they \emph{cannot collapse to the boundary} so that proving Proposition \ref{PropWeakPwEst} does not require new arguments with respect to the ones to get \cite[Proposition 4.1]{DMMT}. In addition, \eqref{ControlLambda1} and \eqref{MuToZero} yield that $\gamma_{i,\varepsilon}\to +\infty$, so that \eqref{Bubbling} really points out a concentration profile at first order for the $u_\varepsilon$'s.

But at that stage, as in \cite[Section 2]{DMMT}, the strategy is to consider a radially symmetric much more precise ansatz solving the same equation (i.e. set now $h_0\equiv 0$ in \cite[(2.6)]{DMMT}), whose relevance in the non-radial setting can be quantified (see \cite[Section 3]{DMMT} picking now $h\equiv 0$ in \cite[(3.8)]{DMMT}). Thus assumption \cite[(2.5)]{DMMT} and the related \cite[(3.6), (3.17)-(3.19),  second requirement in (3.21), (4.26), Lemma 4.1]{DMMT} can be now ignored. The terms $w_{l,\varepsilon}$ in \cite[(4.47)]{DMMT} and $\tilde{w}_\varepsilon$ in \cite[(4.85)]{DMMT} used to control this additional linear term are also now useless and we eventually do get Theorem \ref{ThmBUp} by following the lines of \cite[Sections 2-4]{DMMT}, since no specific new difficulty may arise close to the boundary thanks to Propositions \ref{PropMovPlane} and \ref{PropWeakPwEst}. As a last important remark, the lower bound on $d(x_{i,\varepsilon},\partial\Omega)$ from Proposition \ref{PropWeakPwEst} is also used to get the key property $\lambda_\varepsilon\to 0^+$, arguing as in \cite[Step 4.2]{DMMT}.
\end{proof}

\begin{rem} As pointed out in \cite{MartThiVet}, the positivity of the $u_\ve$'s is essential to have energy quantization (i.e. $\beta\in 4\pi\N^\star$) in Theorem \ref{ThmBUp}. Moreover, even restricting to nonnegative functions, we stress that there exist Palais-Smale sequences associated to \eqref{MainEq} with arbitrary limiting energy $\beta>4\pi$ and converging weakly to $0$ in $H^1_0$ (see Costa-Tintarev \cite{CostaTintarev}), so that not only the quantization, but even the quantification proved for solutions by Druet \cite{DruetDuke} fail for general Palais-Smale sequences (see also \cite{ThiMan2} for examples with non-zero weak limit entering in the framework of \cite{DruetDuke}, but where \emph{quantization fails}). This is in striking constrast with a large class of otherwise closely related critical problems for which quantification already holds for Palais-Smale sequences (see for instance the pioneering work \cite{Struwe} by Struwe). This is maybe the clearest evidence of the huge difficulty to run directly standard variational techniques generating Palais-Smale sequences to solve \eqref{MainEq}, and motivates our compactness techniques considering only \emph{exact solutions} instead of \emph{Palais-Smale sequences}. 
\end{rem}

In the specific case of the mean-field equation, i.e. when $p_\varepsilon=1$ for all $\varepsilon$, with the same notation and assumptions as in Theorem \ref{ThmBUp}, we obviously get 
\begin{equation}\label{ModifiedBetaMeanField}
\beta_\varepsilon^{MF}:=\frac{\lambda_\varepsilon}{2}\int_\Omega e^{u} dx=\beta_\varepsilon+\frac{\lambda_\varepsilon|\Omega|}{2}\to\beta=4\pi N\in 4\pi \mathbb{N}^\star
\end{equation}
as $\varepsilon\to 0$. At that stage, another additional ingredient for our proof is the following result stated here with the same notation and conventions as in Theorem \ref{ThmDegree}:

\begin{thm}\label{ThmChenLin}
Let $\Omega\subset \mathbb{R}^2$ be a smoothly bounded domain and $h\in C^1(\bar\Omega)$ be positive. Let $N\in \mathbb{N}$ and $\beta\in (4\pi N,4\pi(N+1))$ be given. Then the total Leray-Schauder degree $d^{MF}_{\beta,h}(\Omega)$ of all the solutions of the mean-field equation
\begin{equation}\label{MeanFieldEquationh}\tag{$\mathcal{E}_{\beta,h}^{MF}$}
\begin{cases}
& \Delta u=\beta\frac{2  h e^u}{\int_\Omega h e^u dx}\,,\\
& u=0 \text{ on }\partial\Omega\,,
\end{cases}
\end{equation} 
is well-defined and equals $\binom{N-\chi(\Omega)}{N}$.
\end{thm}

As an obvious remark, a solution of \eqref{MeanFieldEquationh} has to be positive in $\Omega$. Theorem \ref{ThmChenLin}, which is only needed here for a constant function $h>0$, has been stated first in Chen-Lin \cite{ChenLin-Liouville}. Indeed, these authors show in their beautiful work:

\begin{prop}[Chen-Lin \cite{ChenLin-Liouville}]\label{p:degreejump} For all $N\in\N$ and all $\beta\in (4\pi N, 4\pi (N+1))$, $d^{MF}_{\beta,h}(\Omega)$ is equal to a constant $d_N^{MF}(\Omega)$, which is $1$ for $N=0$. Moreover, we have the following degree jump formula:
\begin{equation}\label{degreejump}
d_{N+1}^{MF}(\Omega)-d_{N}^{MF}(\Omega)= \frac{(-1)^N}{N!}\deg(\nabla\Phi_{N,h},\Omega^N\setminus D_N,0),
\end{equation}
where
\begin{equation}\label{defPhi}
  \Phi_{N,h}(x_1,\dots,x_N) = 8 \pi \sum_{ i \neq j} \mathcal{G}(x_i,x_j) + 4\pi \sum_{1\le i\le N}  
  \mathcal{H}(x_i,x_i) +\sum_{1\le i\le N} \ln h(x_i)
\end{equation}
is defined in $\Omega^N\setminus D_N$ with $D_N:=\{(x_1,\dots,x_N)\in \Omega^N: \exists i\neq j\,, ~x_i=x_j\}$, where $\mathcal{G}$ is the Green's function of $\Delta$ with zero Dirichlet boundary conditions on $\Omega$, where $\mathcal{H}$ is the regular part of $\mathcal{G}$ normalized here as
\begin{equation}\label{ConventionRegPartGreen}
  \mathcal{G}(x,y) = \frac{1}{2 \pi}  \ln \frac{1}{|x-y|} + \mathcal{H}(x,y), 
\end{equation}
and where $\deg(\nabla\Phi_{N,h},\Omega^N\setminus D_N,0)$ is the Brouwer degree (or total index) of the vector field $\nabla \Phi_{N,h}$.
\end{prop}

Contrary to what occurs on a closed manifold, it turns out that, in order to compute the degree of $\nabla \Phi_{N,h}$, the classical Poincaré-Hopf formula does \emph{not} directly apply in the present setting of a bounded domain since $\nabla \Phi_{N,h}$ points outwards only on a strict subset of the boundary of its domain, non-empty if $N>1$. Another goal of this paper is to provide an argument showing that, for topological reasons relying strongly on the fact that every connected component of $\partial \Omega$ has zero Euler characteristic in our two-dimensional case (crucial point in Lemma \ref{l:deg-in-Theta-I}, see also Remark \ref{Odd} just below), the following formula for the degree of $\nabla\Phi_{N,h}$ holds in the present setting of a smoothly bounded domain $\Omega$, so that the degree formula in Theorem \ref{ThmChenLin} follows from Propositions \ref{p:degreejump} and \ref{p:degree}.

\begin{rem}\label{Odd}
It is known that any closed odd-dimensional manifold has zero Euler-characteristic, so that this aforementioned property for every connected boundary component of $\partial \Omega$ holds more generally for any smoothly bounded $\Omega \subset \mathbb{R}^n$ with $n$ \emph{even}. Indeed, with minimal changes, our proof extends to this case which is relevant for instance for the fourth-order mean-field equation in Lin-Wei-Wang \cite{LinWeiWang11}.
\end{rem}

\begin{prop}\label{p:degree}
The degree of $\nabla\Phi_{N,h}$ is well-defined on $ \Omega^N \setminus D_N$ and one has 
\begin{equation}\label{eq:degree-formula}
 \deg(\nabla \Phi_{N,h},\Omega^N\setminus D_N,0 ) = \chi (\Omega) (\chi (\Omega) -1) ... (\chi (\Omega) -N + 1). 
\end{equation}
\end{prop}
\noindent Proving Proposition \ref{p:degree} is the purpose of Section \ref{SectChenLin}. 

\medskip

We are now in position to prove Theorems \ref{MainThm} and \ref{ThmDegree}, starting with the specific case where $\beta>0$ is out of the set $4\pi \mathbb{N}^\star$ of the critical energy levels.

\begin{proof}[Proof of the degree formula of Theorem \ref{ThmDegree} for $\beta\not \in 4\pi\mathbb{N}^\star$] During this whole first part, we fix $N\in \mathbb{N}$ and $\beta\in (4\pi N, 4\pi (N+1))$. For $t\in [0,1]$ and $p\in [1,2]$, let $T_{p,\beta}$ and $T_{t,\beta}^{MF}$ be the Fredholm-type nonlinear operators given for any non-negative non-zero function $v\in C^1(\bar{\Omega})$ by 
 $$T_{p,\beta}(v)=v-\Delta^{-1}\left(\frac{2\beta v^{p-1} e^{v^p}}{p \left(\int_\Omega \left(e^{v^p}-1\right) dx\right)^{\frac{2-p}{p}}\left(\int_\Omega v^p e^{v^p} dx \right)^{\frac{2(p-1)}{p}}} \right) $$
 and 
 $$ T_{t,\beta}^{MF}(v)=v-\Delta^{-1}\left(\frac{2\beta e^v}{\int_\Omega (e^u-t) dx} \right)\,, $$
where $\Delta^{-1}f$ is the solution $w$ of $\Delta w=f$ with zero Dirichlet boundary conditions. By Theorem \ref{ThmBUp} and formula \eqref{ModifiedBetaMeanField} with standard elliptic theory and \eqref{ControlLambda1}, there exist constants $C>c>0$ such that 
\begin{equation*}
\text{$\|u\|_{C^1(\bar{\Omega})}< C$ and $u> c v_1$ in $\Omega$}
\end{equation*}
 for all $u$ positive in $\Omega$ solving $T_{t,\beta}^{MF}(u)=0$ for some $t\in [0,1]$, and all $u$ solving $T_{p,\beta}(u)=0$, i.e. \eqref{IntermEq}, for some $p\in [1,2]$, where $v_1$ is still a given first Dirichlet eigenfunction of $\Delta$ chosen positive in $\Omega$. Let $C^1_0(\bar{\Omega})$ be the closed subspace of $C^1(\bar{\Omega})$ consisting of the functions vanishing on $\partial\Omega$. Let $\mathcal{V}$ be the open subset of $C^1_0(\bar{\Omega})$ given by 
$$\mathcal{V}=\left\{\quad v\in C^1_0(\bar{\Omega})\quad\text{ s.t. }
\begin{array}{l}
v>c v_1\text{ in }\Omega\,,\\
\|v\|_{C^1(\bar{\Omega})}< C
\end{array}
 \qquad\right\}\,. $$ 
Using Hopf's lemma, observe that only positive functions in $\Omega$ are in the closure of $\mathcal{V}$ for the $C^1(\bar{\Omega})$-topology, in particular neither the zero nor any sign-changing function is in this closure. Then, 
$$\text{ $(p,v)\mapsto T_{p,\beta}(v)$ and $(t,v)\mapsto T_{t,\beta}^{MF}(v)$}$$
 are continuous from $[1,2]\times \mathcal{V}$ and $[0,1]\times \mathcal{V}$ respectively into $C^1_0(\bar{\Omega})$, by standard elliptic theory. Observe in particular that this continuity does hold for $T_{p,\beta}$ up to $p=1$ despite the jump of $(p,t)\in [1,2]\times (0+\infty) \mapsto pt^{p-1} e^{t^p}$ at $t=0^+$ as $p\to 1$. Using also that these operators do not vanish on the boundary of $\mathcal{V}$ by construction, the total Leray-Schauder degree $d_{p,\beta}(\Omega):=\mathrm{deg}_{LS}(T_{p,\beta},\mathcal{V},0)$ (resp. $d_{t,\beta}^{MF}(\Omega):=\mathrm{deg}_{LS}(T_{t,\beta}^{MF},\mathcal{V},0)$) of all the solutions of \eqref{IntermEq} (resp. of all the functions $v>0$ in $\Omega$ such that $T_{t,\beta}^{MF}(v)=0$) is well-defined, see for instance Nirenberg \cite{Nirenberg}.

\begin{rem}
It is convenient to work here within the space $C^1_0(\bar{\Omega})$, but one could ask whether the degree thus defined coincides with the one that one could define within the natural variational space $H^1_0$. The arguments in Li \cite[Appendix B]{PrescSnLi} slightly transposed to the present situation show that it is actually the case. 
\end{rem} 
 
\noindent  Now, for $t=0$, the formula for the total degree $d_{0,\beta}^{MF}(\Omega)$ of all the solutions of \eqref{MeanFieldEquation} is given by Theorem \ref{ThmChenLin}. Moreover, we clearly have $T_{1,\beta}=T_{1,\beta}^{MF}$, so that $d_{1,\beta}(\Omega)=d_{1,\beta}^{MF}(\Omega)$. Using then our definition of $\mathcal{V}$, Theorem \ref{ThmChenLin} and the homotopy invariance of the total Leray-Schauder degree as far as compactness holds, we have $d_{p,\beta}(\Omega)=d_{1,\beta}(\Omega)$ for all $p\in [1,2]$ on the one hand, while $d_{t,\beta}^{MF}(\Omega)=d_{0,\beta}^{MF}(\Omega)=d_{1,\beta}^{MF}(\Omega)$ for all $t\in [0,1]$ on the other hand, which concludes the proof of Theorem \ref{ThmDegree}, and thus that of Theorem \ref{MainThm}, in this first case where $\beta\not \in 4\pi\mathbb{N}^\star$.  
\end{proof}

\begin{rem}\label{RemBupSol}
Let $p\in [1,2]$ be fixed. A posteriori, by contraposing the homotopy invariance of the Leray-Schauder degree as far as compactness holds, observe that the effective degree jump of $d_{p,\beta}(\Omega)$ at the levels $\beta\in 4\pi \mathbb{N}^\star$, as soon as the genus $g(\Omega)$ is greater than $1$, imposes the existence of a sequence of blowing-up solutions with limiting energy $\beta=4\pi N$ for all $N\in \mathbb{N}^\star$ as in Theorem \ref{ThmBUp} with $p_\varepsilon=p$ (see also Deng-Musso \cite{DengMusso} or del Pino-Musso-Ruf \cite{DelPNewSol,DelPBeyond} for specific constructions of such blowing-up solutions for \eqref{IntermEq}).
\end{rem}

In order to complete the proof of Theorems \ref{MainThm} and \ref{ThmDegree}, it remains to treat the case $\beta\in 4\pi \N^\star$. This will be based on the estimate in the following theorem, saying roughly that for $p\in (1,2]$ the amount of Dirichlet energy near each blow-up point approaches $4\pi$ \emph{from above}. This estimate was first proven in the radial case and for $p=2$ by Mancini and the second author \cite{MartMan}, then extended to the case of a bounded domain in $\R^2$ by the third author \cite{WhenExtremals} (see also Ibrahim-Masmoudi-Nakanishi-Sani \cite{MasNakSan} for related results), finally on a closed two-dimensional surface \cite[Theorem 5.1]{DMMT}. Taking Proposition \ref{PropMovPlane} into account, the proof of \cite[Theorem 5.1]{DMMT} can be easily adapted to yield:

\begin{thm}\label{t:4pi+} Assume $p_\ve= p\in (1,2]$ for all $\varepsilon$ and let $u_\ve$, $\lambda_\ve$, $\beta_\ve\to \beta =4\pi N\in 4\pi \N^\star$ be as in Theorem \ref{ThmBUp}. Then
\begin{equation}\label{eq:4pi+}
\beta_\ve\ge 4\pi\left( N+\frac{4(p-1)(1+o(1))}{p^2}\sum_{1\le i\le N} \frac{1}{\gamma_{i,\ve}^{2p}}\right)\quad \text{as }\ve\to 0
\end{equation}
where the sequence $(\gamma_{i,\varepsilon})_\varepsilon$ is given by Proposition \ref{PropWeakPwEst} for all $i$.
\end{thm}

\begin{proof}[Conclusion of the proof of Theorems \ref{MainThm} and \ref{ThmDegree}] We fix now $p\in (1,2]$. Then, from Theorem \ref{t:4pi+} we get that $\beta_\varepsilon>\beta$ for all $\varepsilon$ small enough in Theorem \ref{ThmBUp}. As a consequence, for any given $N\in \mathbb{N}$ and small $\eta\in (0,1]$, we have compactness of all the solutions of \eqref{IntermEq} for all $\beta\in [4\pi (N+\eta), 4\pi (N+1)]$. In particular, we may define as above an open neighborhood $\mathcal{V}'\subset C^1_0(\bar{\Omega})$ containing all these solutions and whose closure contains only positive functions in $\Omega$.  Then, the map $(\beta, v)\mapsto T_{p,\beta}(v)$ is continuous from $[4\pi (N+\eta), 4\pi (N+1)]\times \mathcal{V}'$ to $C^1_0(\bar{\Omega})$ and, by a deformation argument as above, we get that the total Leray-Schauder degree $\mathrm{deg}_{LS}(T_{p,\beta},\mathcal{V}',0)$ of all the solutions of \eqref{IntermEq} is well-defined and does not depend on $\beta\in [4\pi (N+\eta), 4\pi (N+1)]$: Theorems \ref{MainThm} and \ref{ThmDegree} are proven.
\end{proof}

\begin{rem}\label{NonAutonomousVersion}
\upshape{Relying now on \cite{DruThiI}, we may get a more general non-homogeneous version of Theorem \ref{MainThm}, which can be stated as follows:
\begin{thm}\label{MainThmh}
Let $\Omega\subset \mathbb{R}^2$ be a smoothly bounded domain and $h$ be a smooth positive $C^2$-function in $\bar{\Omega}$. Let $F_h$ be given in $H^1_0$ by
$$ F_h(v):=\int_\Omega \left(e^{v^2}-1 \right) h~ dx $$
and $\mathcal{M}_\beta\subset H^1_0$ be as in \eqref{DefMBeta}. Assume that $\Omega$ is not contractible. Then, given any positive real number $\beta>0$, there exists a nonnegative function $u$, critical point of $F_h$ constrained to $\mathcal{M}_\beta$. In particular, $u$ is smooth  and solves \eqref{MainEqf} below for $p=2$.
\end{thm}
\begin{proof}[Proof of Theorem \ref{MainThmh}] We fix $p=2$. Given $\beta\in (4\pi N , 4\pi (N+1)]$, we claim that the total Leray-Schauder degree $d_{p,\beta,h}(\Omega)$ of the solutions of 
\begin{equation}\label{MainEqf}\tag{$\mathcal{E}_{p,\beta,h}$}
\begin{cases}
&\Delta u= p\lambda h u^{p-1} e^{u^p}\,,\, u>0\,\text{ in }\Omega\,,\\
&u=0\text{ on }\partial \Omega\,,\\
&\frac{\lambda p^2}{2} \left(\int_\Omega \left(e^{u^p}-1\right) h~dx\right)^{\frac{2-p}{p}}\left(\int_\Omega u^p e^{u^p} h~dx \right)^{\frac{2(p-1)}{p}}=\beta\,,
\end{cases}
\end{equation}
is well-defined and does not depend of $h$, positive $C^2$-function in $\bar{\Omega}$, so that it always equals $\binom{N-\chi(\Omega)}{N}$, formula given by Theorem \ref{ThmDegree} for $h\equiv 1$. Indeed, first given $\beta>0$ out of $4\pi \mathbb{N}^\star$, it follows from the quantization and the associated compactness result in \cite{DruThiI} that $d_{p,\beta,h_r}(\Omega)$ is well-defined and does not depend of $r\in [0,1]$, where $h_r=(1-r)+r h$. Fixing now such an $h$, $N\in \mathbb{N}$ and $\eta\in (0,1]$ and using that boundary blow-up cannot either occur for the solutions of
\eqref{MainEqf} (see Remark \ref{BdryBup} just below), the energy estimate in Theorem \ref{t:4pi+} from \cite[Section 5]{DMMT} still holds here and gives the compactness of all the solutions of \eqref{MainEqf} for all $\beta\in [4\pi(N+\eta), 4\pi (N+1)]$, so that $d_{p,\beta,h}(\Omega)$ does not depend on $\beta$ in this range by following the second part of the above proof: Theorem \ref{MainThmh} is proven.
\end{proof}
}
\end{rem}

\begin{rem}\label{BdryBup}
Excluding boundary blow-up is actually the most delicate part in \cite{DruThiI} where $p$ is fixed equal to $2$. The analogue property appears to be open for $p\in (1,2)$.
\end{rem}

\section{Degree formula for the vector field associated to \ref{MeanFieldEquationh} on a bounded domain $\Omega$}\label{SectChenLin}

The purpose of this section is to prove Proposition \ref{p:degree}. Let $\mathcal{G}$ be the Green's function of $\Delta$ with zero Dirichlet boundary conditions and $\mathcal{H}$ be its regular part with the convention in \eqref{ConventionRegPartGreen}. Given $y\in \Omega$, notice that
$$
  \mathcal{G}(x,y) \to + \infty \quad \hbox{ as } x \to y  \qquad\text{ and that} \qquad 
  \mathcal{H}(x,x) \to - \infty \quad \hbox{ as } x \to \partial \Omega. 
$$
Given a positive function $h\in C^1(\overline{\Omega})$, let $\Phi_{N,h}$ be given by \eqref{defPhi}. 
To fix the notation, we will write
$$\nabla \Phi_{N,h}(x_1,\dots,x_N)=(\nabla_{x_1}\Phi_{N,h}(x_1,\dots,x_N),\dots, \nabla_{x_N}\Phi_{N,h}(x_1,\dots,x_N)),$$
with
\begin{equation*}\label{nablaPhiN}
\begin{split}
\nabla_{x_j}\Phi_{N,h}(x_1,\dots,x_N) &=
16\pi \sum_{\scriptstyle 1\le \ell\le N\atop \scriptstyle \ell\ne j} \nabla \M{G}(x_j,x_\ell) + 8\pi \nabla\M{H}(x_j,x_j) +\frac{\nabla h(x_j)}{h(x_j)},
\end{split}
\end{equation*}
where we denote by $\nabla\M{G}(x,y)$ and $\nabla\M{H}(x,y)$ the gradient with respect to the first variable.

\medskip

\noindent In order to get Proposition \ref{p:degree}, we first need to analyse the compactness of the critical points of $\Phi_{N,h}$.

\subsection{Compactness of critical points of $\Phi_{N,h}$}\label{ss:cpt}

In this subsection we prove that, uniformly with respect to $t\in [0,1]$, critical the points of
$$\Phi_{N,h}^t(x_1,\dots,x_N) := 8 \pi \sum_{1\le i,j\le N\atop i \neq j} \mathcal{G}(x_i,x_j) + 4\pi \sum_{1\le i\le N}  
  \mathcal{H}(x_i,x_i) +t\sum_{1\le i\le N} \ln h(x_i)$$
cannot be arbitrarily close 
to the boundary of $\Omega^N\setminus D_N$, implying that the degree in Proposition \ref{p:degree} is well-defined. More precisely, define 
\begin{equation}\label{OmegaNdelta}
\begin{split}
(\Omega^N\setminus D_N)_\delta:=\{(x_1, \dots, x_N)\in \Omega^N\setminus D_N: &\, |x_i-x_j|> \delta \text{ for }1\le i<j\le N,\\
& d(x_i,\de\Omega)> \delta\text{ for }1\le i\le N \}.
\end{split}
\end{equation}

\begin{lem}\label{l:distance}
For every $M>0$, there exists $\delta > 0$ such that $$|\nabla \Phi_{N,h}^t(x_1,\dots,x_N)|\ge M,\quad \text{for } (x_1,\dots, x_N)\in (\Omega^N\setminus D_N)\setminus (\Omega^N\setminus D_N)_\delta, \; t\in [0,1].$$
As a consequence, the degree of $\nabla \Phi_{N,h}^t$ is well-defined on $\Omega^N \setminus D_N$ as
$$\deg(\nabla\Phi_{N,h}, \Omega^N\setminus D_N,0):=\deg(\nabla\Phi_{N,h}, A,0),$$ 
for any open set $A$ with $(\Omega^N\setminus D_N)_{\delta}\subset A \subset\Omega^N\setminus D_N$.
\end{lem}

\begin{proof}
Consider sequences
$$(t_k)_{k\in\mathbb{N}}\subset[0,1],\quad (x_{1,k},\dots,x_{N,k})_{k\in\mathbb{N}}\subset\Omega^N\setminus D_N$$
such that
\begin{equation}\label{nablaPhitk}
\nabla \Phi_{N,h}^{t_k}(x_{1,k},\dots,x_{N,k})=O(1),\quad \text{as }k\to\infty.
\end{equation}
Set
\begin{equation}\label{rik}
r_{i,k}:=\min\left\{d_{i,k},\min_{j\ne i}|x_{i,k}-x_{j,k}|\right\},\quad d_{i,k}:=d(x_{i,k},\de\Omega).
\end{equation}
Up to extracting a subsequence and reordering the indices we can assume that
\begin{equation}\label{rik2}
r_{1,k}\le r_{j,k}, \quad \text{for } 1 \le  j\le N,
\end{equation}
and set
\begin{equation}\label{xkrk}
r_k:=r_{1,k},\quad x_k:=x_{1,k},\quad d_k:= d_{1,k}.
\end{equation}
It suffices to prove that $\liminf_{k\to\infty} r_k>0$. Then we assume, by contradiction, that
\begin{equation}\label{rik0}
\lim_{k\to\infty} r_{k}=0.
\end{equation}
We now scale the domain $\Omega$ and the functional $\Phi_{N,h}^{t_k}$ at $x_k$ by a factor $r_k$. More precisely, we set $\Omega_{k}:=r_k^{-1}(\Omega-x_k)$ and, for $(z_1,\dots, z_N)\in \Omega_k^N\setminus D_N$, we define
\begin{equation*}
\begin{split}
\Xi_{k}( z_1,\dots, z_N)&=\Phi_{N,h}^{t_k}(x_k+r_k z_1,\dots, x_k+r_k z_N)\\
&=8\pi\sum_{\scriptstyle 1\le i,j\le N \atop \scriptstyle i\ne j}  \mathcal{G}_{k}(z_{i},z_{j})
+4 \pi \sum_{1\le i\le N} \mathcal{H}_{k}(z_{i},z_{i}) +t_k\sum_{1\le i\le N}\ln h_{k}(z_{i}),
\end{split}
\end{equation*}
where $\M{G}_{k}$ and $\M{H}_{k}$ are the Green's function of $\Delta$ on $\Omega_k$ and its regular part, and $h_{k}(z):= h(x_{k}+r_{k}z)$.
Then \eqref{nablaPhitk} gives
\begin{equation}\label{nablaXik}
\nabla \Xi_k(z_{1,k},\dots, z_{N,k})=O(r_k),\quad \text{as }k\to\infty,
\end{equation}
where
$$(z_{1,k},\dots, z_{N,k}):=(r_k^{-1}(x_{1,k}-x_k), \dots, r_k^{-1}(x_{N,k}-x_k)).$$
Up to a further subsequence we set
\begin{equation}\label{Di}
\M{D}:=\{i\in\{1,\dots,N\}: |x_{i,k}-x_{k}|=O(r_{k})\text{ as }k\to\infty\},
\end{equation}
and have
\begin{equation}\label{Di2}
\lim_{k\to\infty}\frac{|x_{i,k}-x_{k}|}{r_k}=+\infty\quad \text{for }i\in \{1,\dots,N\}\setminus \M{D}.
\end{equation}
Clearly $\M{D}\neq\emptyset$, since $1\in \M{D}$. Define also, up to a subsequence,
\begin{equation*}
Z'=(z_i)_{i\in\mathcal{D}},\quad Z'_k=(z_{i,k})_{i\in\mathcal{D}}\to Z'_\infty=(z_{i,\infty})_{i\in\mathcal{D}},\quad \text{as }k\to\infty.
\end{equation*}
Notice that \eqref{rik} and \eqref{Di} imply
\begin{equation}\label{xikxjk}
\frac{1}{C}r_{k}\le |x_{i,k}-x_{j,k}|\le C r_{k}\quad \text{for }i,j \in \M{D}, \quad i\ne j,
\end{equation}
and in particular $z_{i,\infty}\ne z_{j,\infty}$ for $i,j\in\M{D}$, $i\ne j$.
We now define the ``reduced'' {functional}
\begin{equation}\label{XiDk}
\begin{split}
\Xi_{\M{D},k}(z_{1},\dots,z_{N}):=& 8\pi\sum_{\scriptstyle i,j\in \M{D}\atop \scriptstyle i\ne j}  \mathcal{G}_{k}(z_{i},z_{j})
+4 \pi \sum_{i\in \M{D}} \mathcal{H}_{k}(z_{i},z_{i}) +t_k\sum_{i\in \M{D}}\ln h_{k}(z_{i})\\
&+16\pi\sum_{i\in\M{D}}\sum_{\scriptstyle 1\le j\le N\atop \scriptstyle j\not\in \M{D}}  \mathcal{G}_{k}(z_{i},z_{j}).
\end{split}
\end{equation}
From \eqref{nablaXik} we get
\begin{equation}\label{nablaZ'}
\nabla_{Z'}\Xi_{\M{D},k}(z_{1,k},\dots,z_{N,k})=O(r_k).
\end{equation}
Taking \eqref{rik0} into account, we now consider 2 cases.

\medskip

\noindent\textbf{Case 1.} Assume that, up to a subsequence,
\begin{equation}\label{dik}
r_k=o(d_k) \quad\text{as }k\to\infty.
\end{equation}
Then $\Omega_{k}$ invades $\R^2$ as $k\to\infty$ and, choosing $K\subset\R^2$ compact containing $z_{i,k}$ for $i\in \M{D}$ and $k$ large, we get from Lemma \ref{lemmaGk} \begin{align*}
\nabla \mathcal{G}_{k}(z_{i,k},z_{j,k})&= \frac{1}{2\pi}\nabla \ln\frac{1}{|z_{i,k}-z_{j,k}|}+O\(\frac{r_{k}}{d_k}\),\quad \text{for }i,j\in\M{D},\\
\nabla \mathcal{G}_{k}(z_{i,k},z_{j,k})&= \frac{1}{2\pi}\nabla \ln\frac{1}{|z_{i,k}-z_{j,k}|}+O\(\frac{r_{k}}{d_k}\)=o(1),\quad \text{for }i\in\M{D}, \,j\not\in \M{D},\\
\nabla \mathcal{H}_{k}(z_{i,k},z_{i,k})&= O\(\frac{r_{k}}{d_k}\)\quad \text{for }i\in\M{D},
\end{align*}
as $k\to\infty$. Moreover $\nabla \ln h_{k}(z_{i,k})=O(r_{k})$ as $k\to\infty$ for $i\in\M{D}$. 
Then from \eqref{XiDk} and \eqref{nablaZ'} we 
infer
\begin{equation}\label{XiIJk2}
\nabla \Xi_{\M{D},\R^2}(Z'_k)=o(1),\quad \text{as }k\to\infty,
\end{equation}
where
\begin{equation}\label{XiJR2}
\Xi_{\M{D},\R^2}(Z'):= 4 \sum_{\scriptstyle i,j\in \M{D}\atop \scriptstyle i\ne j} \ln\frac{1}{|z_{i}-z_{j}|}.
\end{equation}
Notice that, thanks to \eqref{rik}, \eqref{Di} and \eqref{dik},  we have $\M{D}\cap \{2,\dots,N\}\ne \emptyset$. Since also $1\in \M{D}$, the sum in \eqref{XiJR2} is non-empty.
Letting now $k\to\infty$ in \eqref{XiIJk2} it follows that $\nabla \Xi_{\M{D},\R^2}( Z'_{\infty})=0$. 
This contradicts Lemma \ref{lemmaXi2}.



\medskip

\noindent\textbf{Case 2.} Assume that we are not in Case 1. Then since $r_k\ge C d_k$ for some $C>0$ and $r_k\le d_k$, we have up to a subsequence
$$\frac{d_k}{r_k}\to L \ge 1,\quad \text{as }k\to\infty.$$
Then, up to a rotation, $\Omega_k$ converges to the half-space $H=\{(x,y)\in \R^2:y<L$\}.  Let $\M{G}_H$ and $\M{H}_H$ be the Green's function on $H$ and its regular part, as given explicitly in the appendix.
Choosing $K\subset H$ compact containing $z_{i,k}$ for $i\in \M{D}$ and $k$ large, by Lemma \ref{lemmaHk} we estimate
\begin{align*}
\nabla \mathcal{G}_{k}(z_{i,k},z_{j,k})&= \nabla \M{G}_H(z_{i,k},z_{j,k})+O\(r_k\),\quad \text{for }i,j\in\M{D}\\
\nabla \mathcal{G}_{k}(z_{i,k},z_{j,k})&=o(1),\quad \text{for }i\in\M{D}, \,j\not\in \M{D}\\
\nabla \mathcal{H}_{k}(z_{i,k},z_{i,k})&= \nabla \mathcal{H}_{H}(z_{i,k},z_{i,k})+ O\(r_{k}\)\quad \text{for }i\in\M{D},
\end{align*}
and $\nabla \ln h_{k}(z_{i,k})=O(r_{k})$ as $k\to\infty$ for $i\in\M{D}$,
hence,
from \eqref{XiDk} and \eqref{nablaZ'} 
we obtain 
\begin{equation}\label{XiIJk}
\nabla \Xi_{ \M{D},H}(Z'_k)=o(1),\quad \text{as }k\to\infty,
\end{equation}
where
$$\Xi_{ \M{D},H}(Z'):= 8 \pi \sum_{\scriptstyle i,j\in  \M{D}\atop \scriptstyle i\ne j}  \mathcal{G}_H(z_i,z_j) +4\pi  \sum_{i\in  \M{D}} \mathcal{H}_H(z_i,z_i).$$
Letting $k\to \infty$ in \eqref{XiIJk} we get $\nabla \Xi_{ \M{D},H}(Z'_\infty) =0,$
contradicting Lemma \ref{LemmaXiH}.

\medskip

\noindent These contradictions, arising from assumptions \eqref{nablaPhitk} and \eqref{rik0}, complete the proof of the lemma.
\end{proof}

A first consequence of Lemma \ref{l:distance} is that the degree of $\Phi_{N,h}$ does not depend on $h$, so that in the following we will work with $h\equiv 1$.

\begin{cor}\label{cordegPhi}
We have
$$\deg(\nabla \Phi_{N,h},\Omega^N\setminus D_N,0)= \deg(\nabla \Phi_N,\Omega^N\setminus D_N,0 ),$$
where, for $(x_1,\dots,x_N)\in \Omega^N\setminus D_N$,
$$\Phi_N(x_1,\dots,x_N):=8 \pi \sum_{\scriptstyle 1\le i,j\le N\atop \scriptstyle i \neq j} \mathcal{G}(x_i,x_j) + 4\pi \sum_{1\le i\le N} 
  \mathcal{H}(x_i,x_i).$$
\end{cor}

\begin{proof} We can deform $\Phi_{N,h}=\Phi^1_{N,h}$ into $\Phi_N=\Phi^0_{N,h}$ and use the invariance of the degree under homotopy, since for $\delta$ sufficiently small and $t\in [0,1]$, $\Phi_{N,h}^t$ has no critical points on $\partial (\Omega^N\setminus D_N)_\delta$.
\end{proof}




\subsection{Bending the Robin function to apply the Poincar\'e-Hopf theorem} \label{ss:pf}

Since we will not be able to apply the theorem of Poincar\'e-Hopf directly to $\nabla\Phi_N$ to compute $\deg(\nabla \Phi_N,\Omega^N\setminus D_N,0)$, as $\nabla\Phi_N$ does not always point outwards near $\partial (\Omega^N\setminus D_N)$, we now introduce a modified function $\tilde{\Phi}_N$ by {\em bending} the function 
$\mathcal{H}(y,y)$ upwards near the boundary of $\Omega$, in order to obtain a function 
that tends to $+\infty$ on $\partial (\Omega^N\setminus D_N)$. More precisely, let $\tau\in C^\infty([0,\infty))$ be a non-decreasing function such that
$$\tau(t)=1\quad \text{for }t\ge \frac12, \quad \tau(t)=-1\quad \text{for }t\le \frac{1}{4},$$
and set 
\begin{equation}\label{defHtilde}
\tilde{\mathcal{H}}(x):=\mathcal{H}(x,x)\tau\left(\frac{d(y,\de\Omega)}{\tilde\delta}\right),
\end{equation}
see Figure \ref{fig:tilde-h}. Let then $\tilde{\Phi}_N$ be given in $\Omega^N\setminus D_N$ by

\begin{figure}[h] 
\centering
\includegraphics[width=0.3\linewidth]{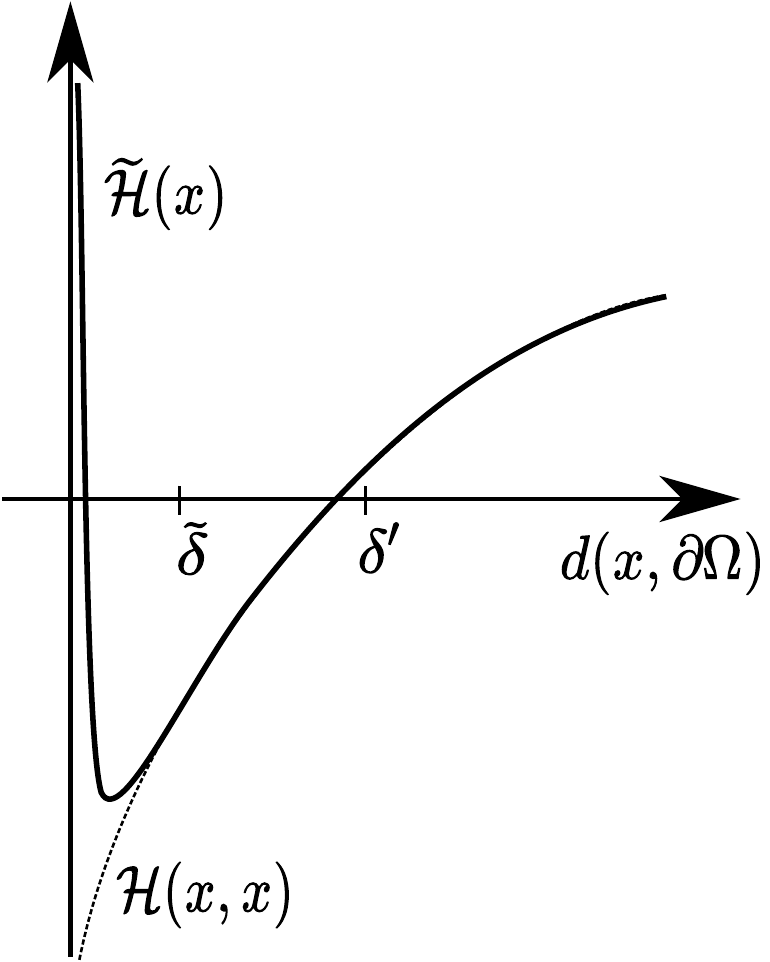}
\caption{Heuristic sketch of $\tilde{\mathcal{H}}$ near $\partial\Omega$}
\label{fig:tilde-h}
\end{figure}

\begin{equation}\label{defPhitilde}
 \tilde{\Phi}_N(x_1,\dots,x_N) = 8 \pi \sum_{\scriptstyle 1\le i,j\le N\atop \scriptstyle i \neq j}  \mathcal{G}(x_i,x_j) + 4\pi \sum_{1\le i\le N} 
 \tilde{\mathcal{H}}(x_i). 
\end{equation}

We next show that for $0<\delta'<\delta/2$ and for $\tilde \delta$ sufficiently small (depending on $\delta'$), the topological strips determined by $\delta'< d(y,\de\Omega)<\delta$ is a {\em forbidden zone} for critical points of $\tilde{\Phi}$, as made precise in the following lemma.

\begin{lem}\label{l:distance-2}
For every $M'>0$ there exist $\delta>0$, such that for any $\delta'\in (0,\delta/2)$ there exists $\tilde\delta\in (0,\delta'/2)$ such that, setting $\tilde{\Phi}_N$ depending on $\tilde \delta$ as in \eqref{defHtilde}-\eqref{defPhitilde}, we have 
$$|\nabla \tilde \Phi_N(x_1,\dots,x_N)|\ge M' $$
if $(x_1,\dots,x_N)\in \Omega^N\setminus D_N$ is such that $\delta'\le d(x_i,\de\Omega)\le \delta$ for some $1\le i\le N$.
In particular, if $(x_1, \dots, x_N)$ is a critical point of $\tilde{\Phi}_N$, then either 
$d(x_i, \de \Omega) > \delta$ or $d(x_i, \de \Omega) < \delta'$ for every $i=1,\dots,N$.
\end{lem}

\begin{proof} 
{Fix} $M'>0$ and let $\delta:=\min_{1\le I\le N}\delta_I$, where $\delta_I$ is the constant $\delta$ appearing in Lemma \ref{l:distance} applied with $N$ replaced by $I$ and with $M=M'+1>0$.

Assuming by contradiction that the lemma is false, we can find  $\delta'\in (0,\delta/2)$, and sequences $\tilde\delta_k\to 0$ and $(x_{1,k},\dots, x_{N,k})\subset\Omega^N\setminus D_N$ such that, for $\tilde \Phi_{N,k}$ defined using $\tilde \delta_k$, we have
\begin{equation}\label{nablatildePhik}
|\nabla \tilde\Phi_{N,k}(x_{1,k},\dots,x_{N,k})|< M
'\quad \text{as }k\to\infty,
\end{equation}
and such that, up to relabelling,
\begin{equation}\label{delta2delta}
\delta'\le d(x_{1,k},\de\Omega)\le\delta.
\end{equation}
Up to passing to subsequence and a further relabelling, let $I\in \{1,\dots,N\}$ such that
\begin{equation}\label{dyik}
\lim_{k\to\infty} d(x_{i,k},\de\Omega)>0 \quad \text{for }1\le i\le I,
\end{equation}
and
$$\lim_{k\to\infty} d(x_{j,k},\de\Omega)=0,\quad I+1\le j\le N.$$
By \eqref{delta2delta} and Lemma \ref{l:distance} we have
\begin{equation}\label{nablaPhiIe}
|\nabla \Phi_I(x_{1,k},\dots, x_{I,k})|\ge M'+1,
\end{equation}
where
$$\Phi_I(x_1,\dots,x_I):=8\pi\sum_{\scriptstyle 1\le i,j\le I\atop \scriptstyle i\ne j}\M{G}(x_i,x_j)+4\pi \sum_{1\le i\le I}\M{H}(x_i,x_i).$$
Observe, on the other, hand that
\begin{equation}\label{nablaGbordo}
\nabla \mathcal{G}(x_{i,k},x_{j,k})=o(1),\quad 1\le i\le I,\quad I+1\le j\le N,
\end{equation}
with $o(1)\to 0$ as $k\to\infty$ (since $x_{j,k}$ is approaching $\de\Omega$, while $x_{i,k}$ is not as $k\to\infty$). Then, by \eqref{dyik}, for $k$ large enough (hence $\tilde\delta_k$ sufficiently small), we 
have
\begin{equation*}
\begin{split}
\tilde\Phi_{N,k}(x_{1,k},\dots, x_{N,k})=&\Phi_I(x_{1,k},\dots, x_{I,k})+16\pi\sum_{\scriptstyle 1\le i\le I\atop \scriptstyle I+1\le j\le N}\M{G}(x_{i,k},x_{j,k})\\
&+8\pi\sum_{I+1\le i, j\le N}\M{G}(x_{i,k},x_{j,k})+4\pi\sum_{I+1\le J\le N}\tilde{\M{H}}_k(x_{j,k}).
\end{split}    
\end{equation*}
Then, setting $X':=(x_1,\dots,x_I)$, from \eqref{nablatildePhik}, \eqref{nablaPhiIe} and \eqref{nablaGbordo} we get
\begin{equation*}
\begin{split}
M'>&|\nabla_{X'}\tilde\Phi_{N,k}(x_{1,k},\dots,x_{N,k})|= |\nabla\Phi_I(x_{1,k},\dots,x_{I,k})|+o(1)\\
\ge& M'+1+o(1),
\end{split}    
\end{equation*}
giving a contradiction for $k$ large enough.
\end{proof}

\noindent We now have the analog of Lemma \ref{l:distance} for $\tilde\Phi_N.$
\begin{lem}\label{l:distance-3} For every $M''>0$ there exists $\delta''\le \tfrac{\tilde \delta}{2}$ depending on $\tilde\delta$ and $M''$ such that 
$$|\nabla \tilde\Phi_N(x_1,\dots, x_N)|\ge M''$$
for every $(x_1,\dots, x_N)\in (\Omega^N\setminus D_N)\setminus (\Omega^N\setminus D_N)_{\delta''}$. In particular, all the critical points of $\nabla \tilde\Phi_N$ are in $(\Omega^N\setminus D_N)_{\delta''}$ and we can define
$$\deg(\nabla\tilde \Phi_{N}, \Omega^N\setminus D_N,0):=\deg(\nabla\tilde \Phi_{N}, A,0),$$ 
for any open set $A$ such that $(\Omega^N\setminus D_N)_{\delta''}\subset A \subset\Omega^N\setminus D_N.$
\end{lem}
\begin{proof}
The proof by contradiction is essentially identical to the proof of Lemma \ref{l:distance}, upon noticing the following facts. In Case $1$, namely $r_k=o(d_k)$, if (up to a subsequence) $d_k\to 0$ as $k\to\infty$, then
$$\tilde{\M{H}}(x_{i,k})=-\M{H}(x_{i,k},x_{i,k})\quad \text{for }i\in\M{D}, \quad  k\text{ large},$$
hence
$$\nabla \tilde{\M{H}}_k(z_{i,k})= - \nabla \M{H}_k(z_{i,k},z_{i,k})=O(r_k),\quad \text{for }i\in\M{D}, \quad  k\text{ large}.$$
If $d_k\not\to 0$, from \eqref{defHtilde}, we bound
$$|\nabla \tilde{\M{H}}(x_{i,k})|=O(\nabla \M{H}(x_{i,k},x_{i,k})+ C_{\tilde\delta} |\M{H}(x_{i,k},x_{i,k})|)= O(1),\quad \text{as }k\to\infty,\quad i\in\M{D},$$
hence again $\nabla \tilde{\M{H}}_k(z_{i,k})=O(r_k)$ as $k\to\infty$ for $i\in\M{D}$, where $\tilde{\M{H}}_k(z):=\tilde{\M{H}}(x_k+r_kz)$. Then, up to a subsequence,  $Z_{k}'$ converges to $Z_\infty'$, a critical point of $\Xi_{\M{D},\R^2}$, contradicting Lemma \ref{lemmaXi2}.

In Case $2$, since $d(x_{1,k},\de\Omega)\to 0$, we also have $d(x_{i,k},\de\Omega)\to 0$ as $k\to\infty$ for $i\in\M{D}$, hence $\tilde{\M{H}}(x_{i,k})=-\M{H}(x_{i,k},x_{i,k})$ for $i\in \M{D}$ and $k$ large enough. Then, setting $\Xi^-_{\M{D},k}$ as in \eqref{XiDk} with $h_k\equiv 1$ and $\M{H}_k(z_i,z_i)$ replaced by $-\M{H}_k(z_i,z_i)$, we obtain
$\nabla \Xi^-_{ \M{D},H}(Z'_\infty) =0,$ where
$$\Xi^-_{ \M{D},H}(Z'):= 8 \pi \sum_{\scriptstyle i,j\in  \M{D}\atop \scriptstyle i\ne j}  \mathcal{G}_H(z_i,z_j) -4\pi  \sum_{i\in  \M{D}} \mathcal{H}_H(z_i,z_i),$$
contradicting Lemma \ref{LemmaXiH}.
\end{proof}

\noindent We recall next a variant of the classical Poincar\'e-Hopf index theorem, which can be found e.g. in \cite{chang}, pages 99-104 (here we adapt the statement to our purposes).


\begin{prop}\label{p:P-H} 
Let $U\subset\R^n$ be an open set and consider
$f \in C^{2}(U,\R)$. 
Assume also that, for some $b\in  \R$, $f^b : = \{x\in U: f(x) < b \}$
has compact closure, and that $f$ has no critical points in $f^{-1}(b)$. Then we have
\begin{equation}\label{eq:LSeuler}
 \deg(\n f, f^b, 0) = \chi(f^b).
\end{equation}
\end{prop}

\begin{lem}\label{l:degtildePhi} We have 
\begin{equation}\label{eq:deg-tilde-Phi}
\deg(\nabla \tilde\Phi_N, \Omega^N \setminus D_N, 0)=\chi (\Omega) (\chi (\Omega) -1) ... (\chi (\Omega) -N + 1).
\end{equation}
\end{lem}

\begin{proof}
Since
$$\lim_{(x_1,\dots,x_N)\to \partial (\Omega^N\setminus D_N)}\tilde \Phi_N(x_1,\dots,x_N)=+\infty,$$ we can take $b>0$ large enough such that
$$(\Omega^N\setminus D_N)_{\delta''}\subset (\tilde{\Phi}_N)^b,$$
where $\delta''$ is as in Lemma \ref{l:distance-3} with $M''=1$.
Then, since $(\Omega^N\setminus D_N)\setminus (\tilde{\Phi}_N)^b$ contains no critical points of $\tilde \Phi_N$ by Lemma \ref{l:distance-3}, we can apply Proposition \ref{p:P-H} to get
$$\deg( \nabla \tilde\Phi_N,\Omega^N\setminus D_N,0)=\deg( \nabla \tilde\Phi_N,(\tilde{\Phi}_N)^b,0)=\chi((\tilde{\Phi}_N)^b).$$
Since $(\Omega^N\setminus D_N)\setminus (\tilde{\Phi}_N)^b$ contains no critical points of $\tilde \Phi_N$, we can use a gradient flow to retract $(\Omega^N\setminus D_N)$ to $(\tilde{\Phi}_N)^b$. Therefore $\chi((\tilde{\Phi}_N)^b)=\chi(\Omega^N\setminus D_N),$
and we conclude with Proposition \ref{p:configurations} below.
\end{proof}

\subsection{Counting the new critical points}
By bending the Robin function we have in general created new critical points near the boundary. In this section we will show that these new critical points do not change the total degree. To this purpose, we now define the following sets (see Figure \ref{fig:strip}):
\begin{equation*}\label{eq:Omega-delta}
 \Omega_\delta:= \left\{ y \in \Omega \; : \; d(y,\de \Omega) > \delta \right\}; \quad \Sigma_{\delta'}:=\Omega\setminus \overline{\Omega}_{\delta'},
\end{equation*}
and for $0\le I \le N $
\begin{equation*}\label{eq:Theta-delta-I}
\begin{split}
& \Theta_{\delta,\delta',I}^* = \left\{ (x_1, \dots, x_N) \in \Omega^N \setminus D_N \; : \;  x_1,\dots,x_I\in {\Omega}_\delta,\; x_{I+1},\dots,x_N\in {\Sigma}_{\delta'} \right\}\\
&\qquad \quad =((\Omega_\delta)^I\setminus D_I)\times ((\Sigma_{\delta'})^{N-I}\setminus D_{N-I});\\
&\Theta_{\delta,\delta',I} = \big\{ (x_1, \dots, x_N) \in \Omega^N \setminus D_N \; : \;  (x_{\sigma(1)},\dots,x_{\sigma(N)})\in   \Theta_{\delta,\delta',I}^*\\
&\qquad \qquad \text{ for a permutation }\sigma\in S_N \big\};\\
&\Theta_{\delta,\delta',\delta'',I}^*= \Theta_{\delta,\delta',I}^*\cap (\Omega^N\setminus D_N)_{\delta''},\quad 
\Theta_{\delta,\delta',\delta'',I}= \Theta_{\delta,\delta',I}\cap (\Omega^N\setminus D_N)_{\delta''}.
\end{split}
\end{equation*}
\begin{figure}
\centering
\includegraphics[width=0.3\linewidth]{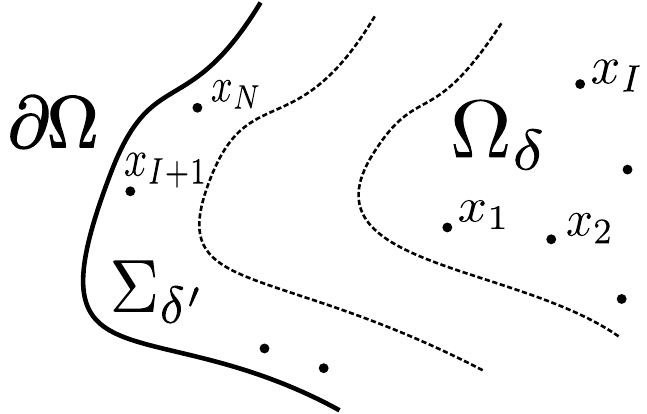}
\caption{$(x_1,\dots, x_N)\in \Theta^*_{\delta,\delta',I}=((\Omega_\delta)^I\setminus D_I)\times ((\Sigma_{\delta'})^{N-I}\setminus D_{N-I}) $}
\label{fig:strip}
\end{figure}

\noindent By Lemmata \ref{l:distance-2} and \ref{l:distance-3}, we can define for $0\le I\le N-1$
$$\deg(\nabla \tilde\Phi_N, \Theta_{\delta,\delta',I}, 0):=\deg(\nabla \tilde\Phi_N, A, 0)$$
for any open set $A$ such that
$\Theta_{\delta,\delta',\delta'',I}\subset A\subset \Theta_{\delta,\delta',I}.$
 Moreover, there are no critical points of $\tilde\Phi_N$ in $(\Omega^N \setminus D_N) \setminus \cup_{I=0}^N \Theta_{\delta,\delta',I}$ by Lemma \ref{l:distance-2}; therefore by the excision property and noting that  $\Theta_{\delta,\delta',N}=(\Omega_\delta)^N\setminus D_N$ and that $\tilde{\Phi}_N=\Phi_N$ inside $(\Omega_\delta)^N \setminus D_N$, we have
\begin{equation}\label{eq:excision}
\begin{split}
\deg(\nabla \tilde\Phi_N, \Omega^N \setminus D_N, 0)&= \deg(\nabla \Phi_N, (\Omega_\delta)^N \setminus D_N, 0) +\! 
\sum_{I=0}^{N-1}\! \deg(\nabla \tilde\Phi_N, \Theta_{\delta,\delta',I}, 0)\\
&\!\!\!\!\!\!\!=\deg(\nabla \Phi_N, (\Omega_\delta)^N \setminus D_N, 0) + 
\sum_{I=0}^{N-1}\binom{N}{I} \deg(\nabla \tilde\Phi_N, \Theta_{\delta,\delta',I}^*, 0).
\end{split}
\end{equation}

\begin{lem}\label{l:deg-in-Theta-I} Up to choosing $\delta'$ sufficiently small in Lemma \ref{l:distance-2} (and $\tilde\delta$ and $\delta''$ correspondingly), we have that  
\begin{equation*}
\begin{split}
 \deg(\n \tilde\Phi_N, \Theta_{\delta,\delta',I}^*, 0)=& \deg(\n_{X'} \Phi_I, (\Omega_\delta)^I \setminus D_I, 0)\\
 &\qquad\qquad \times \deg(\n_{X''} \tilde{\Phi}_{N-I}, (\Sigma_{\delta'})^{N-I} \setminus D_{N-I}, 0),
\end{split}
\end{equation*}
for $0\le I\le N-1$, where
$$X'=(x_1,\dots, x_I),\quad X''=(x_{I+1},\dots, x_N),\quad X=(X',X''),$$
and the degrees are well-defined intersecting the domains with $(\Omega^I\setminus D_I)_{\delta''}$ and $(\Omega^{N-I}\setminus D_{N-I})_{\delta''}$.
\end{lem}

\begin{proof}
For $X\in \Theta_{\delta,\delta',I}^*$ we have
\begin{equation}\label{yiya}
  x_i \in \Omega_\delta \quad \hbox{for }   1\le i\le I, \qquad
  x_a \in \Sigma_{\delta'} \quad \hbox{for }  I+1\le a\le N. 
\end{equation}
Interestingly, the gradient almost {\em decouples} into the first $I$ components and the last $N-I$. In fact, 
since $\tilde{\mathcal{H}}(x)$ coincides with $\mathcal{H}(x,x)$ for $x\in \Omega_\delta$, 
we have
\begin{equation*}
\tilde\Phi_N (X) = \Phi_I(X') +\tilde \Phi_{N-I}(X'')+\Xi_{I,N}(X),
\end{equation*}
where
\begin{equation*}
\begin{split}
\tilde \Phi_{N-I}(X'')&= 8\pi \sum_{\scriptstyle I+1\le a,b\le N\atop \scriptstyle a\ne b} \mathcal{G}(x_a, x_b)+4\pi \sum_{I+1\le a\le N} \tilde{\mathcal{H}}(x_a);\\
\Xi_{I,N}(X)&= 16\pi \sum_{1\le i\le I<a\le N} \mathcal{G}(x_i, x_a).
\end{split}
\end{equation*}
Notice that
$$\mathcal{G}(x_i, x_a)=o_{\delta'}(1),\quad  \nabla \mathcal{G}(x_i, x_a)=o_{\delta'}(1),$$
with $o_{\delta'}(1)\to 0$ as $\delta'\to 0$ (for $\delta$ fixed), uniformly with respect to $x_i$ and $x_a$ as in \eqref{yiya}.
Now observe that $X=(x_1,\dots, x_N)\in \partial \Theta_{\delta,\delta',\delta'',I}^*$ (with $\delta''<\tfrac{\tilde\delta}{2}$ to be fixed) implies
that one of the following holds
\begin{enumerate}
\item $|x_i-x_j|=\delta''$ for some $1\le i<j\le I$;
\item $|x_a-x_b|=\delta''$ for some $I+1\le a<b\le N$;
\item $d(x_i,\de\Omega)=\delta$ for some $1\le i\le I$;
\item $d(x_a,\de\Omega)=\delta'$ for some $I+1\le a\le N$;
\item $d(x_a,\de\Omega)=\delta''$ for some $I+1\le a\le N$.
\end{enumerate}
If we choose $M=M'=M''=1$ in Lemmata \ref{l:distance}, \ref{l:distance-2} and \ref{l:distance-3} and $\delta, \delta',\delta''$ sufficiently small, we obtain that in all the above cases
\begin{equation}\label{nPhi>1}
\max\{|\nabla_{X'} \Phi_I(X')|,|\nabla_{X''} \tilde{\Phi}_{N-I}(X'')|\}\ge 1, \quad \text{for }X\in \partial \Theta_{\delta,\delta',\delta'',I}^*.
\end{equation}
Then, by the analog of \eqref{nablaGbordo}, we can choose $\delta'$ even smaller (changing also $\tilde{\delta}$ accordingly) so that in addition to \eqref{nPhi>1} we also have
\begin{equation}\label{nXi<1}
|\nabla \Xi_{I,N}(X)|\le \frac{1}{2} \quad \text{for }X\in \partial \Theta_{\delta,\delta',\delta'',I}^*.
\end{equation}
Consider now for $t\in [0,1]$
$$\tilde\Phi_N^t (X) := \Phi_I(X') +\tilde \Phi_{N-I}(X'')+t\Xi_{I,N}(X).$$
Then \eqref{nPhi>1}-\eqref{nXi<1} imply at once
$$|\nabla \tilde \Phi_N^t(X)|\ge |(\nabla_{X'}\Phi_I(X'),\nabla_{X''}\tilde{\Phi}_{N-I}(X'')) |-|\nabla \Xi_{I,N}(X)|>0 \text{ for } X\in \partial \Theta_{\delta,\delta',\delta'',I}^*.$$
By the invariance of the degree under homotopy, we obtain 
\begin{equation}
\begin{split}
\deg(\n \tilde\Phi_N, \Theta_{\delta,\delta',\delta'',I}^*, 0) &=\deg(\n \tilde{\Phi}_{N}^0, \Theta_{\delta,\delta',\delta'',I}^*, 0)\\
&= \deg((\nabla_{X'}\Phi_I(X') , \nabla_{X''}\tilde{\Phi}_{N-I}(X'')), \Theta_{\delta,\delta',\delta'',I}^*, 0),
\end{split}
\end{equation}
and we conclude by the product formula of the degree.
\end{proof}

\begin{lem}\label{l:degtildePhi2} Up to choosing $\delta'$ sufficiently small (reducing $\tilde \delta,\delta''$ accordingly), for $1\le J\le N$ we have
$$\deg(\n \tilde{\Phi}_J, (\Sigma_{\delta'})^J \setminus D_J, 0)=0.$$
\end{lem}

\begin{proof} We will work on the the larger domain $\Sigma_{2\delta'}$, 
and modify $\tilde{\M{H}}$ in the strip $\Sigma_{2\delta'}\setminus \Sigma_{\delta'}=\bar{\Omega}_{\delta'}\setminus \bar{\Omega}_{2\delta'}$ by defining
$$\hat{\M{H}}(x):=\tilde{\M{H}}(x)+ \sigma\left(\frac{d(x,\de\Omega)}{\delta'}\right),$$
where $\sigma\in C^{\infty}([0,2))$ is non-negative, non-decreasing, and
$$\sigma(t)=\begin{cases}
0& \text{for }0\le t\le 1,\\
\ln\(\frac{1}{2-t}\)& \text{for }\frac{3}{2}\le t<2.
\end{cases}
$$
\begin{figure}
\centering
\includegraphics[width=0.3\linewidth]{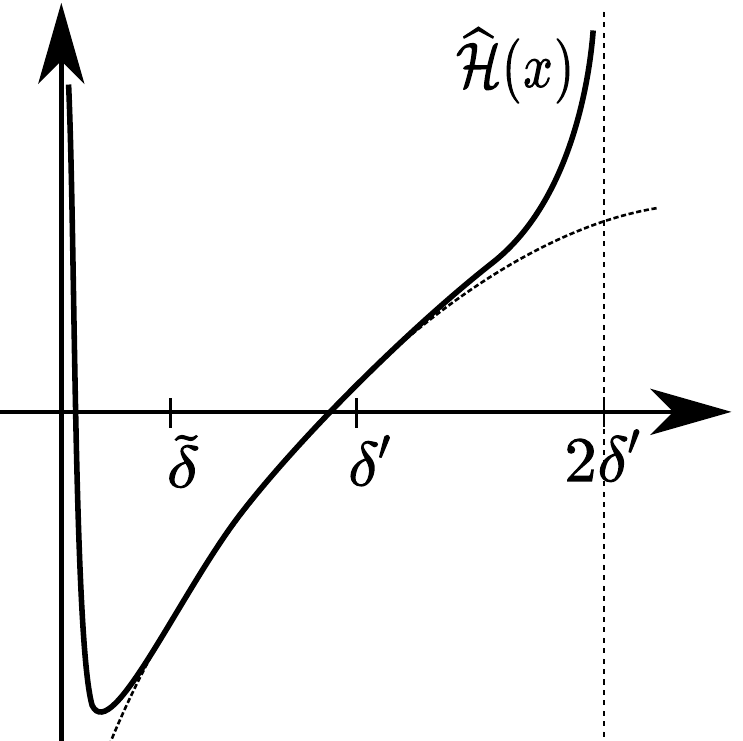}
\label{fig:strip2}
\end{figure}
Clearly $\hat{\M{H}}=\tilde{\M{H}}$ in $\Sigma_{\delta'}$, so that
\begin{equation}\label{deghatPhi0}
\deg(\n \tilde\Phi_J, (\Sigma_{\delta'})^J \setminus D_J, 0)=\deg(\n \hat{\Phi}_J, (\Sigma_{\delta'})^J \setminus D_J,0),
\end{equation}
where, for $(x_1,\dots, x_J)\in (\Sigma_{2\delta'})^J\setminus D_J$,
\begin{equation}\label{hatPhi}
\begin{split}
\hat\Phi_J(x_1,\dots, x_J)&=\tilde \Phi_J(x_1,\dots, x_J)+\sum_{1\le i\le J} \sigma\left(\frac{d(x_i,\de\Omega)}{\delta'}\right)\\
&=8\pi \sum_{\scriptstyle 1\le i,j\le J\atop \scriptstyle i\ne j}\M{G}(x_i,x_j)+4\pi\sum_{1\le i\le J}\hat{\M{H}}(x_i).
\end{split}
\end{equation}
Thanks to Lemmata \ref{l:distance-4} and \ref{l:distance-5} below, we can choose $\delta'>0$ even smaller than before, and $\tilde\delta$ accordingly small (without affecting the previous results, which are valid for $\delta'$,  $\tilde\delta$ and $\delta''$ sufficiently small), so that $\hat\Phi_J$ has no critical points in $\left((\Sigma_{2\delta'})^J\setminus(\Sigma_{\delta'})^J\right)\setminus D_J\,.$ Then
\begin{equation}\label{deghatPhi}
\deg(\n \hat{\Phi}_J, (\Sigma_{\delta'})^J \setminus D_J,0)= \deg(\n \hat{\Phi}_J, (\Sigma_{2\delta'})^J \setminus D_J,0).
\end{equation}
Since now
$$\lim_{(x_1,\dots,x_J)\to \partial ((\Sigma_{2\delta'})^J\setminus D_J)}\hat\Phi_J(x_1,\dots,x_J)=+\infty,$$
we can apply Proposition \ref{p:P-H}, and get
\begin{equation}\label{deghatPhi2}
 \deg(\n \hat\Phi_J, (\Sigma_{2\delta'})^J \setminus D_J, 0)  = \chi((\Sigma_{2\delta'})^J \setminus D_J).
\end{equation}
On the other hand, $\chi(\Sigma_{2\delta'})=0$, since $\Sigma_{2\delta'}$  retracts to a union of circles (up to choosing $\delta'$ sufficiently small) and $\chi(\mathbb{S}^1)=0$. Then it follows from 
Proposition \ref{p:configurations} 
$$\chi((\Sigma_{2\delta'})^J \setminus D_J)=0.$$
This, together with \eqref{deghatPhi0}, \eqref{deghatPhi} and \eqref{deghatPhi2} yields $\,\deg(\n \tilde\Phi_J, (\Sigma_{\delta'})^J \setminus D_J, 0)=\chi((\Sigma_{2\delta'})^J \setminus D_J)=0.$
\end{proof}

\begin{lem}\label{l:distance-4} For $1\le J\le N$ and $M'>0$ and $\delta>0$ there exists $\delta'\in (0,\delta/2)$ such that
$$|\nabla \Psi_J|\ge M' \quad \text{in }(\Sigma_{2\delta'})^J \setminus D_J,$$
where, for $x_1,\dots, x_J\in (\Omega\setminus \overline{\Omega}_{2\delta'})^J \setminus D_J$,
\begin{equation}\label{defPsi}
\begin{split}
\Psi_J(x_1,\dots, x_J)&:=\Phi_J(x_1,\dots, x_J)+\sum_{1\le i\le J} \sigma\left(\frac{d(x_i,\de\Omega)}{\delta'}\right)
\\
&=8\pi \sum_{\scriptstyle 1\le i,j\le J\atop\scriptstyle i\ne j}\M{G}(x_i,x_j)+4\pi\sum_{1\le i\le J}\M{H}(x_i)+\sum_{1\le i\le J} \sigma\left(\frac{d(x_i,\de\Omega)}{\delta'}\right),
\end{split}
\end{equation}
and $\M{G}$ and $\M{H}$ are still the Green's function of $\Delta$ in $\Omega$ (not in $\Omega\setminus \overline{\Omega}_{2\delta'}$) and its regular part.
\end{lem}

\begin{proof}
We assume by contradiction that there exist sequences $\delta'_k\to 0^+$ and
$$(x_{1,k},\dots, x_{J,k})\in (\Sigma_{2\delta_k'})^J\setminus D_J$$
such that
\begin{equation}\label{nablaPsiJ}
\nabla \Psi_J^{k}(x_{1,k},\dots, x_{J,k})=O(1)\quad \text{as }k\to\infty,
\end{equation}
where the notation $\Psi_J^{k}$ is meant to emphasize that $\Psi_J$ depends on $\delta_k'$, due to the last term in \eqref{defPsi}. We proceed as in the proof of Lemma \ref{l:distance}.
Set
$$r_{i,k}=\min\left\{d_{i,k},\hat d_{i,k}, \min_{j\ne i}|x_{i,k}-x_{j,k}|\right\},$$
where
$$d_{i,k}:=d(x_{i,k},\de\Omega), \quad  \hat d_{i,k}:=d(x_{i,k},\de\Omega_{2\delta_k'})=2\delta_k'-d(x_{i,k},\de\Omega).$$
Up to reordering we assume $r_{1,k}\le r_{i,k}$ for every $1\le i\le J$, and set
$$r_k:= r_{1,k}, \quad d_k:=d_{1,k}, \quad \hat d_k:= \hat d_{1,k}, \quad x_k:= x_{1,k}.$$
Define $\M{D}$, $Z_k$, $Z'_k\to Z'_\infty$, $\Omega_k$, as in the proof of Lemma \ref{l:distance}, and similarly
$$(\Sigma_{2\delta_k'})_k:=r_k^{-1}(\Sigma_{2\delta_k'}-x_k).$$
Then we set, for $Z=(z_1,\dots,z_J)\in ((\Sigma_{2\delta_k'})_k)^J\setminus D_J$,
\begin{equation}\label{reducedsigma}
\begin{split}
\hat \Xi_{\M{D},k}(z_{1},\dots,z_{N}):=&8\pi \sum_{\scriptstyle i,j\in \M{D} \atop \scriptstyle i\ne j} \mathcal{G}_{k}(z_{i},z_{j})+ 16\pi\sum_{i\in \M{D}} \sum_{\scriptstyle 1\le j\le N \atop \scriptstyle j\not\in\M{D}} \mathcal{G}_{k}(z_{i},z_{j})\\
&+4 \pi \sum_{i\in \M{D}} \mathcal{H}_{k}(z_{i},z_{i}) +\sum_{i\in \M{D}}\sigma \(\frac{r_k d(z_{i},\partial \Omega_k)}{\delta'_k}\),
\end{split}
\end{equation}
where $\M{G}_k$ and $\M{H}_k$ are the Green's function of $\Delta$ on $\Omega_k$ and its regular part, respectively. Observe that $r_k\le \delta_k'\to 0$ as $k\to\infty$. We consider several cases.

\medskip

\noindent\textbf{Case 1} Assume that, up to a subsequence, $r_k=o(d_k)$, $r_k=o(\hat d_k)$ as $k\to\infty$. 
Since
$$\nabla_z \sigma \(\frac{r_k d(z,\partial \Omega_k)}{\delta'_k}\)= \nabla_z \sigma \(\frac{d(x_k+r_k z,\partial \Omega)}{\delta'_k}\)=r_k \nabla_x \sigma \(\frac{d(x,\partial \Omega)}{\delta'_k}\)\bigg|_{x=x_k+r_kz},$$
and
$$\sigma'\(\frac{d(x_{i,k},\partial \Omega)}{\delta'_k}\)=O\(\frac{\delta_k'}{\hat d_k}\),\quad |\nabla d(\cdot, \partial\Omega)|=1,\quad \text{for }i\in \M{D}$$
we infer
$$\nabla_z \sigma \(\frac{r_k d(z_{i,k},\partial \Omega_k)}{\delta'_k}\)=O\(\frac{r_k}{\hat d_k}\)=o(1)\quad \text{as }k\to\infty,\quad \text{for }i\in \M{D}.$$
Then we can proceed as in Case 1 of the proof of Lemma \ref{l:distance} to get that $Z_\infty'$ is critical point of $\Xi_{\M{D},\R^2}$, as defined in \eqref{XiJR2}, contradicting Lemma \ref{lemmaXi2}.

\medskip

\noindent\textbf{Case 2} Assume that, up to a subsequence, $r_k=o(\hat d_k)$, $\frac{d_k}{r_k}\to L\ge 1$ as $k\to\infty$. This implies that $r_k=o(\delta_k')$ as $k\to\infty$ and, up to a rotation, $(\Sigma_{\delta_k'})_k$ converges to a half-space $H=\{(x,y)\in \R^2:y<L\}$. We proceed as in Case 1, except that now 
$$\sigma \(\frac{r_k d(z_{i,k},\partial \Omega_k)}{\delta'_k}\)=0,\quad \text{for }i\in \M{D}, \quad k\text{ large enough}.$$
Then $\hat\Xi_{\M{D},k}(Z_k)=\Xi_{\M{D},k}(Z_k)$ for $k$ large enough, where $\Xi_{\M{D},k}$ is as in \eqref{XiDk} (with $t_k=0$). Then we are in the same situation as Case 2 of the proof of Lemma \ref{l:distance}, leading to the same contradiction, i.e.  $Z_\infty'$ is a critical point of $\Xi_{\M{D},H}$, violating Lemma \ref{LemmaXiH}.

\medskip

\noindent\textbf{Case 3} Assume that, up to a subsequence, $r_k=o(d_k)$, $\frac{\hat d_k}{r_k}\to \hat L_1\ge 1$. This implies in particular that, up to a subsequence,
\begin{equation}\label{rkhatL}
r_k=o(\delta_k'),\quad  \frac{\hat d_{i,k}}{r_k}\to \hat L_i\ge 1\quad \text{as }k\to\infty, \text{ for }i\in \M{D}.
\end{equation}
Then, up to a rotation
$(\Sigma_{2\delta_k'})_k$ converges to $\hat H=\{(x,y)\in \R^2:y>-\hat L\}$.
As before, $r_k=o(d_k)$ implies $r_k=o(d_{i,k})$ for $i\in\M{D}$ and it follows $\nabla \M{H}_k(z_{i,k},z_{i,k})=O(r_k)$ as $k\to\infty$ for $i\in \M{D}$. Considering that
$$\frac{r_k d(z_{i,k},\partial \Omega_k)}{\delta_k'}= \frac{ d(x_{i,k},\partial \Omega)}{\delta_k'}\to 2,\quad \text{as }k\to\infty, \quad \text{for }i\in \M{D},$$
and $\sigma(t)=-\ln(2-t)$ for $\frac{3}{2}\le t<2$, we have that
\begin{equation}
\begin{split}
\nabla_z\sigma\(\frac{r_k d(z_{i,k},\de \Omega_k)}{\delta_k'}\)&
=\frac{r_k\nabla d(z_{i,k},\partial \Omega_k)}{2\delta_k'-r_k d(z_{i,k},\de \Omega_k)}=\frac{-r_k\nabla d(z_{i,k},\partial (\Omega\setminus\Omega_{2\delta_k'})_k)}{\hat d_{i,k}}\\
&\to -\frac{\nabla_z d(z_{i,\infty},\partial \hat H)}{\hat L_i}= \nabla_z \ln\frac{1}{d(z_{i,\infty},\de\hat H)},
\end{split}
\end{equation}
since $\hat L_i=d(z_{i,\infty},\de\hat H).$
Then $Z_\infty'$ is a critical point of 
$$\hat\Xi_{\M{D},\hat H}(Z)= 4\sum_{\scriptstyle i,j\in\M{D}\atop \scriptstyle  i\ne j} \ln\frac{1}{|z_i-z_j|}+\sum_{i\in\M{D}}\ln\frac{1}{d(z_i, \de\hat H)},$$
contradicting Lemma \ref{lemmaXi3}.

\medskip

\noindent\textbf{Case 4} If we are not in any of the above cases, then, up to a subsequence,
$$\frac{d_k}{r_k}\to L_1\ge 1,\quad \frac{\hat d_k}{r_k} \to \hat L_1\ge 1, \quad \text{as }k\to\infty.$$
Then,
$$\frac{r_k}{\delta_k'}\to \gamma>0,\quad \text{as }k\to\infty,$$
{and},
$$\frac{d_{i,k}}{r_k}\to L_i\ge 1,\quad \frac{\hat d_{i,k}}{r_k} \to \hat L_i\ge 1, \quad \frac{\hat d_{i,k}}{\delta_k'}\to \rho_i\in (0,2),\quad \text{as }k\to\infty, \quad i\in\M{D}.$$
Up to a rotation $(\Sigma_{2\delta_k'})_k$ converges to
$H\cap \hat H$,
where
$$H= \{(x,y)\in \R^2: y<L_1\}, \quad \hat H= \{(x,y)\in \R^2: y>-\hat L_1\}.$$
Considering now that
$$L_i+o(1)=\frac{d_{i,k}}{r_k}=d(z_{i,k},\de\Omega_k)=d(z_{i,\infty},\de H)+o(1),\quad \text{as }k\to\infty, \quad\text{for }i\in\M{D},$$
we compute
\begin{equation*}
\begin{split}
\nabla_z\sigma \(\frac{r_k d(z_{i,k},\partial \Omega_k)}{\delta'_k}\)&
=\frac{r_k}{\delta'_k}\sigma'\(\frac{d_{i,k}}{\delta'_k}\)\nabla_z d(z_{i,k},\partial \Omega_k)\\
&\to \gamma\sigma'(\gamma L_i)\nabla_z d(z_{i,\infty},\de H)\\
&=\nabla_z \sigma(\gamma d(z_{i,\infty},\de H)),\quad \text{as }k\to\infty,\quad \text{for }i\in \M{D}.
\end{split}
\end{equation*}
Then $Z'_\infty$ is a critical point of
$$\hat\Xi_{\M{D},\hat H}(Z')= 8\pi\sum_{\scriptstyle i,j\in\M{D}\atop \scriptstyle i\ne j} \M{G}_H(z_i,z_j)+4\pi \sum_{i\in\M{D}}\M{H}_H(z_i)+ \sum_{i\in\M{D}}\sigma(\gamma d(z_i,H)).$$
This contradicts Lemma \ref{LemmaXiHH}.

\medskip

\noindent These contradictions, arising from \eqref{nablaPsiJ}, complete the proof of the lemma.
\end{proof}

\begin{lem}\label{l:distance-5} Given $1\le J\le N$ and $\delta'$ as in Lemma \ref{l:distance-4} (or smaller), there exists $\tilde \delta\in (0,\tfrac{\delta'}{2})$ sufficiently small such that $\hat\Phi_J$ (defined as in \eqref{hatPhi}, hence depending on $\tilde\delta$) has no critical points in $((\Sigma_{2\delta'})^J\setminus (\Sigma_{\delta'})^J)\setminus D_J.$
\end{lem}

\begin{proof} This follows as in the proof of Lemma \ref{l:distance-2} by a contradiction argument, using Lemma \ref{l:distance-4} instead of Lemma \ref{l:distance}. Here $2\delta'$ plays the role of $\delta$ in Lemma \ref{l:distance-4}.
\end{proof}

\subsection{Proof of Proposition \ref{p:degree} (completed)} 

By Corollary \ref{ControlLambda1} we can assume $h\equiv 1$ and work with $\Phi_N$ instead of $\Phi_{N,h}$. Using Lemma \ref{l:distance}, considering that $\Phi_N$ coincides with $\tilde\Phi_N$ on $(\Omega_\delta)^N \setminus D_N$, using formula \eqref{eq:excision} and  Lemmata \ref{l:deg-in-Theta-I} and \ref{l:degtildePhi2} we now get
\begin{equation*}
\begin{split}
\deg(\n \Phi_N, \Omega^N \setminus D_N, 0) &= \deg(\n \Phi_N, (\Omega_\delta)^N \setminus D_N, 0)\\
 & = \deg(\n \tilde\Phi_N, (\Omega_\delta)^N \setminus D_N, 0) \\
& = \deg(\n \tilde\Phi_N, \Omega^N \setminus D_N, 0) -
   \sum_{0\le I\le N-1} \deg(\n \tilde\Phi_N, \Theta_{\delta,\delta',I}, 0)\\
&   = \deg(\n \tilde\Phi_N, \Omega^N \setminus D_N, 0).
\end{split}
\end{equation*}
By Lemma \ref{l:degtildePhi} we finally obtain
$$\deg(\n \Phi_N, \Omega^N \setminus D_N, 0) = \chi (\Omega) (\chi (\Omega) -1) ... (\chi (\Omega) -N + 1).$$
\hfill$\square$

\appendix
\section{Properties of the Green's function}\label{a:green}

For some fixed $L,\hat L>0$, let $H$ and $\hat H$ denote the half spaces
\begin{equation}\label{defHH}
H=\{(x,y)\in \R^2:y<L\},\quad \hat H=\{(x,y)\in \R^2:y>-\hat L\}\,.
\end{equation}
On the half-space $H$, the Green's function and its regular part are given explicitly by
\begin{equation}\label{GreenH}
\mathcal{G}_H( z,\eta)=\frac{1}{2\pi}\ln\frac{| z-\eta^*|}{| z-\eta|},\quad \mathcal{H}_H( z,\eta)=\frac{1}{2\pi}\ln| z-\eta^*|,
\end{equation}
where $\eta^*$ is the reflection of $\eta$ across $\de H$.

Consider now $\Omega\subset\R^2$ smoothly bounded. For $x_k\in \Omega$, $d_k:=d(x_k,\partial \Omega)$, $r_k>0$, let $\M{G}_k(z,\eta)$ and $\mathcal{H}_k( z,\eta)$ denote the Green's function of $\Delta$ in $\Omega_k:= r_k^{-1}(\Omega-x_k)$ and its regular part.

\begin{lem}\label{lemmaGk} With the above notation, assume that, up to a subsequence, $r_k=o(d_k)$, so that $\Omega_k$ invades $\R^2$ as $k\to\infty$. Then, for every compact set $K\subset\R^2$, there exists $C_K$ such that
$$\left|\nabla_z\(\M{G}_k(z,\eta)-\frac{1}{2\pi}\ln\frac{1}{|z-\eta|}\)\right|\le C_K\frac{r_k}{d_k}$$
and
$$|\nabla_z \M{H}_k(z,z)|\le C_K\frac{r_k}{d_k},$$
uniformly for $z\in K$, $\eta\in \Omega_k$.
\end{lem}

\begin{proof}
This follows from Appendix B in \cite{DruThiI}.
\end{proof}

\begin{lem}\label{lemmaHk}
With the above notation, assume that, up to a subsequence and a rotation, $r_k\to 0$, $d_k/r_k\to L\in (0,\infty)$, and $\Omega_k$ converges to the half-space $H$ as $k\to \infty$. Then for every compact set $K\subset H$ and for $k$ large enough
$$\|\nabla(\M{H}_k-\M{H}_H)\|_{L^\infty(K\times \Omega_{k,\rho})}\le C_K r_k\max\{1,\rho^{-1}\},$$
and
$$\|\nabla(\M{G}_k-\M{G}_H)\|_{L^\infty(K\times \Omega_{k,\rho})}\le C_K r_k\max\{1,\rho^{-1}\},$$
where
$\Omega_{k,\rho}:=\{z\in  \Omega_k:d(z,\de\Omega_k)>\rho\}.$
\end{lem}

\begin{proof}
We have
\begin{equation}
\left\{\begin{array}{ll}
\Delta_ z \mathcal{H}_k( z,\eta)=0&\text{in }\Omega_k\times\Omega_k\\
\mathcal{H}_k( z,\eta)=\frac{1}{2\pi}\ln| z-\eta|&\text{for } z\in \de\Omega_k,\;y\in \Omega_k.
\end{array}
\right.
\end{equation}
We apply the maximum principle to the function
$$\M{H}_k(z,\cdot)-\M{H}_H(z,\cdot):\Omega_k\to \R,$$
where $z\in K$ and $K\subset H$ is a fixed compact set.
We want to prove that 
\begin{equation}\label{Hbordo}
\sup_{z\in K,\eta\in \de \Omega_k} |\M{H}_k(z,\eta)-\M{H}_H(z,\eta)|\le C(K)r_k,
\end{equation}
for $k$ sufficiently large.
By definition we have
$$\M{H}_k(z,\eta)=\frac{1}{2\pi}\ln|z-\eta|,\quad \M{H}_H(z,\eta)=\frac{1}{2\pi}\ln|z-\eta^*|\quad \text{for }\eta\in \de\Omega_k.$$
Let $\delta>0$ be so small that $\de\Omega\cap B_\delta(p(x_k))$ can be written as the graph over $[-\delta,\delta]$ of a function $f$, where $p(x_k)$ is the nearest point projection of $x_k$ onto $\partial\Omega$. Let $\M{L}\subset\R^2$ be the tangent line to $\de\Omega$ at $p(x_k)$ and $f:\M{L}\cap B_\delta(x_k)\to \R$. Up to a rotation we can assume that $f'(p(x_k))=0$ and by a Taylor expansion we obtain
$$|f(t)|\le \frac{\sup_{[-\delta,\delta]}f''}{2}t^2,\quad t\in [-\delta,\delta].$$
In particular for each $x\in \de\Omega\cap B_\delta(x_k)$ we have $d(x,L)\le C_{\Omega,\delta}|x-x_k|^2$.
Scaling to $\Omega_k$ we then obtain
$$\eta\in \de\Omega_k, \quad |\eta|\le \frac{\delta}{r_k} \quad \Rightarrow\quad d(\eta,\de H)=O(r_k).$$
Then, for such $\eta$ and for $z\in K$, we have $|z-\eta|\le |z-\eta^*|+O(r_k)$ and
$$|\M{H}_k-\M{H}_H|=\left|\ln\frac{|z-\eta|}{|z-\eta^*|}\right|\le \ln\left(1+\frac{O(r_k)}{|z-\eta^*|}\right)= C(K)O(r_k),$$
since $\frac{1}{|z-\eta^*|}\ge \ve(K)>0$ uniformly for $z\in K$. On the other hand
$$\eta\in \de\Omega_k, \quad |\eta|\ge \frac{\delta}{r_k} \quad \Rightarrow\quad |z-\eta^*|-|z-\eta|\le C(K).$$
Then
$$|\M{H}_k-\M{H}_H|=\left|\ln\frac{|z-\eta|}{|z-\eta^*|}\right|\le \ln\left(1+\frac{r_k C(K)}{\delta}\right)=  C(K)O(r_k),$$
so that \eqref{Hbordo} is proven. By the maximum principle we obtain
\begin{equation}\label{Hint}
\sup_{z\in K,\eta \in \Omega_k} |\M{H}_k(z,\eta)-\M{H}_H(z,\eta)|= C(K)O(r_k),
\end{equation}
and the bound on $\nabla(\M{H}_k-\M{H}_H)$ follows by elliptic estimates, since both $\M{H}_H(z,\cdot)$ and $\M{H}_k(z,\cdot)$ are harmonic.
The estimate of $\nabla (\M{G}_k-\M{G}_H)$ follows at once since $\M{G}_k-\M{G}_H= \M{H}_k-\M{H}_H$.
\end{proof}

We state Lemmata \ref{lemmaXi2}, \ref{lemmaXi3} and \ref{LemmaXiH} below with general positive coefficients $\alpha_{i,j}$ and $\beta_i$ for future reference, and also to remark that the coefficients $8\pi$ and $4\pi$ appearing in the definitions of $\Phi_{N,h}$, $\Phi_N$, etc..., can be replaced by any other positive constants.

\begin{lem}\label{lemmaXi2}
Consider for some $J\ge 2$ the functional $\Xi_{J,\R^2}: (\R^2)^J\setminus D_J\to \R$
$$\Xi_{J,\R^2}( z_1,\dots, z_J):=\sum_{\scriptstyle 1\le i, j\le J\atop \scriptstyle i\ne  j}\alpha_{i,j} \ln\frac{1}{|z_i-z_j|},$$
with $\alpha_{i,j}>0$ for $1\le i,j\le J$, $i\ne j$. Then $\Xi_{J,\R^2}$ has no critical points.
\end{lem}

\begin{proof}
Consider $1\le i\le J$ such that $ z_i$ lies on the boundary of the convex hull of $\{ z_1,\dots, z_J\}$, and observe that the derivative with respect to $z_i$ in a direction pointing outwards of the convex hull is negative, as $|z_i- z_j|$ is increasing for every $1\le i\ne j\le J$. 
\end{proof}

\begin{lem}\label{lemmaXi3}
Consider for some $J\ge 1$ the functional $\hat\Xi_{J,\hat H}: \hat H^J\setminus D_J\to \R$
$$\hat \Xi_{J,\hat H}( z_1,\dots, z_J):=\sum_{\scriptstyle 1\le i, j\le J\atop \scriptstyle i\ne  j}\alpha_{i,j} \ln\frac{1}{|z_i-z_j|}+\sum_{1\le i\le J}\gamma_i\ln \frac{1}{d(z_i,\partial\hat H)},$$
with $\alpha_{i,j}>0$ for $1\le i<j\le J$ and $\gamma_i> 0$ for $1\le i\le J$. Then
$\hat \Xi_{J,\hat H}$ has no critical point.
\end{lem}

\begin{proof}
Given $(z_1,\dots, z_J)$, write $z_i=(a_i,b_i)$ and up to reordering, assume that $b_1\le \dots\le b_J$, i.e. $z_J$ is one of the points farthest from $\partial \hat H$. Then
$$\frac{\de}{\de b_J}\ln\frac{1}{|z_i-z_J|}\le 0,\quad \text{for }1\le i<J,\quad \frac{\de}{\de b_J}\ln\frac{1}{d(z_J,\de\hat H)}<0,$$
hence
$$\frac{\de \hat \Xi_{J,\hat H}( z_1,\dots, z_J)}{\de b_J}  <0.$$
\end{proof}

\begin{lem}\label{LemmaXiH}
Consider for some $J\ge 1$ the functionals $\Xi_{J,H}^\pm: H^J\setminus D_J\to \R$ given as
$$\Xi_{J,H}^\pm (z_1,\dots, z_J):=\sum_{ \scriptstyle 1\le i,j\le J\atop \scriptstyle i\ne j}\alpha_{i,j} \mathcal{G}_H( z_i, z_j)\pm \sum_{1\le i\le J}\beta_j\mathcal{H}_H( z_i,z_i),$$
with $\alpha_{i,j}>0$ for $1\le i,j\le J$, $i\ne j$, $\beta_j> 0$ for $1\le i\le J$. Then $\Xi_{J,H}^{\pm}$ has no critical points.
\end{lem}

\begin{proof}
Given $(z_1,\dots, z_J)$, write in coordinates $z_i=(a_i,b_i)\in H$ for $1\le i\le J$ and, up to reordering, assume first that $J>1$, $a_1\le a_2\le\dots\le a_J$ and $a_1<a_J$.
Then, it follows from \eqref{GreenH}
$$\frac{\de \M{G}_H(z_1,z_J)}{\de a_1}>0, \quad \frac{\de \M{G}_H(z_1,z_j)}{\de a_1}\ge 0, \text{ for }2\le j\le J-1,\quad \frac{\de \M{H}_H(z_1,z_1)}{\de a_1}=0,$$
hence $\nabla_{z_1} \Xi_{J,H}^\pm( z_1,\dots, z_J)\ne 0$.

If $J=1$ or all the $z_i's$ are vertically, i.e. $a_1=\dots=a_J$, then, up to reordering, assume that $b_1< \dots< b_J$. Then, again from \eqref{GreenH}, we get
\begin{equation}\label{deb1}
\frac{\de \M{G}_H(z_1,z_j)}{\de b_1}<0,\quad \text{for }2\le j\le J, \quad \frac{\de \M{H}_H(z_1,z_1)}{\de b_1}<0,
\end{equation}
hence $\nabla_{z_1} \Xi_{J,H}^+(z_1,\dots, z_J)\ne 0$,
and 
$$\frac{\de \M{G}_H(z_J,z_j)}{\de b_J}>0,\quad \text{for }2\le j\le J, \quad \frac{\de \M{H}_H(z_J,z_J)}{\de b_J}<0,$$
hence $\nabla_{z_N} \Xi_{J,H}^-(z_1,\dots, z_J)\ne 0$.
\end{proof}

\begin{lem}\label{LemmaXiHH}
Consider for some $J\ge 1$ the functional $\Xi_{J,H,\hat H}^\varphi: (H\cap\hat H)^J\setminus D_J\to \R$ given as
$$\Xi_{J,H,\hat H}^\varphi(z_1,\dots,z_J):=\Xi^+_{J,H}(z_1,\dots,z_J)+\sum_{1\le i\le J} \gamma_i\varphi(z_i),$$
where $H$ and $\hat H$ are as in \eqref{defHH} for some $L,\hat L>0$,
$\Xi^+_{J,H}$ is as in Lemma \ref{LemmaXiH},  $\gamma_i\ge 0$, for $1\le i\le J$, $\varphi\in C^1(H\cap \hat H)$ and $\nabla \varphi(z)\cdot(0,1)\le 0$ for every $z\in H\cap \hat H$.
Then $\Xi_{J,H,\hat H}^\varphi$ has no critical points.
\end{lem}

\begin{proof}
The proof is identical to the proof of Lemma \ref{LemmaXiH}, using \eqref{deb1} together with
$$\frac{\de \varphi(z_1)}{\de b_1}=\nabla\varphi(z_1)\cdot (0,1)\le 0.$$
\end{proof}


\section{The Euler characteristic of configuration spaces}\label{a:configuration}

\noindent We will need the following general fact about {\em configuration spaces}.

\begin{prop}\label{p:configurations}
Let $\Omega$ be a smoothly bounded domain of $\R^2$ and call for any $N\in \mathbb{N}^\star$
$$F (\Omega, N):=\Omega^N\setminus D_N=\left\{ (x_1, \dots, x_N) \in \Omega^N  \; : \; x_i \neq x_j \hbox{ for } i \neq j \right\}.$$ Then
$$
  \chi (F (\Omega, N)) = \chi (\Omega) (\chi (\Omega) -1) ... (\chi (\Omega) -N + 1).
  $$
\end{prop}

\begin{proof}
This result follows from the fact that $F(\Omega,N)$ fibers on $F(\Omega,N-1)$ 
with fiber $\Omega_{N-1}$ (projecting on the last of the  $N$ points), where we denote by $\Omega_j$ the set $\Omega$ with $j$ points removed, see \cite[Chapter 2]{Spa66}. 

\medskip

\noindent\textbf{Claim} The above fibration is {\em orientable}, in the sense that $\pi_1(F(\Omega,N-1))$ acts trivially on the homology $H_*(\Omega_{N-1},\R)$ of the fiber.

\begin{proof}[Proof of the claim.]
Since $\Omega_{N-1}$ is homotopically equivalent to a 
(connected) finite union of circles, the only relevant homology group of $\Omega_{N-1}$ is $H_{1}(\Omega_{N-1},\R)$. Consider now a closed path $\gamma\in C^0([0,1], F(\Omega,N-1))$ representing a homotopy class $[\gamma]\in \pi_1(F(\Omega,N-1))$. Set
$$(x_1(t),\dots,x_{N-1}(t)):=\gamma(t).$$
For the point $\gamma(0)=:(x_1,\dots,x_{N-1})\in F(\Omega,N-1)$ consider the fiber
$$\Omega\setminus\{x_1,\dots,x_{N-1}\}\simeq \Omega_{N-1}$$ and choose generators of its first homology group as follows. 
Since $\Omega$ is a two-dimensional domain, it is homotopic to a disk with $r$ points removed, 
for $r = 0, 1, 2, \dots$. Given the $i$-th hole of $\Omega$, $i = 1, \dots, r$, 
choose an element $\sigma_i$ of $H_1(\Omega,\R)$ in the form of a smooth Jordan curve that does not contain any of the points $x_1(t),\dots, x_{N-1}(t)$ with $t \in [0,1]$.   Then it is clear that the action of $\gamma$ on $\sigma_1,\dots,\sigma_r$ is trivial.

Now, for each $x_j$, $j=1,\dots,k-1$, choose a small and smooth Jordan  curve $\eta_j$ around $x_j$ not intersecting $\sigma_1,\dots, \sigma_r$, for instance the oriented boundary of $B_\ve(x_j)$. By compactness, if we choose $\ve$ small enough we get
$$\overline{B_\ve(x_j(t))}\cap \overline{B_\ve(x_i(t))}=\emptyset,\quad \text{for }1\le i< j\le N-1,\;t\in [0,1]$$
and
$$\overline{B_\ve(x_j(t))}\cap \sigma_i=\emptyset,\quad \text{for }1\le j\le N-1,\;1\le i\le r,\;t\in [0,1].$$
Again it follows immediately that the action of $[\gamma]$ of $\eta_j$ is trivial.
Since
$$H_1(\Omega_{N-1},\R)=\text{span}\{\sigma_i, \eta_j, 1\le i\le r, 1\le j\le N-1\}$$
the claim is proven.
\end{proof}

By the product formula for orientable fibrations, see Spanier \cite[p.481]{Spa66}, we infer
$$\chi(F(\Omega,N)) =\chi(F(\Omega,N-1))\chi(\Omega_{N-1})=\chi(F(\Omega,N-1))(\chi(\Omega)-N+1).$$
Now the proposition follows at once by induction in $N$.
\end{proof}

\bibliographystyle{abbrv}
 \bibliography{Bibliography}

\end{document}